\documentclass[12pt]{amsart}
\usepackage{tikz,array, verbatim}
\usepackage[letterpaper, margin = 1in]{geometry}
\usepackage{amsfonts, amsmath, latexsym, epsfig, caption}
\usepackage{amssymb, color}

\usepackage{epsf}

\usepackage{array}
\usepackage{ragged2e}
\usepackage{hyperref}
\usepackage{textgreek}

\DeclareMathOperator{\lin}{lin}
\DeclareMathOperator{\conv}{conv}

\usepackage[foot]{amsaddr}
\usepackage{orcidlink}

\DeclareRobustCommand{\SkipTocEntry}[5]{}

\title[Voronoi conjecture for five-dimensional parallelohedra]{Voronoi conjecture for five-dimensional parallelohedra}

\begin{document}

\author[A. Garber]{Alexey Garber}
\address{Alexey Garber \orcidlink{0000-0002-9474-2077}, School of Mathematical and Statistical Sciences, The University of Texas Rio Grande Valley, Brownsville, TX, USA}
\email{alexey.garber@utrgv.edu}


\newcommand{\R}{\ensuremath{\mathbb{R}}}
\newcommand{\Z}{\ensuremath{\mathbb{Z}}}
\newcommand{\F}{\ensuremath{\mathbb{F}}}

\newtheorem{theorem}{Theorem}[section]
\newtheorem{proposition}[theorem]{Proposition}
\newtheorem{corollary}[theorem]{Corollary}
\newtheorem{lemma}[theorem]{Lemma}
\newtheorem{problem}[theorem]{Problem}
\newtheorem*{conjecture}{Conjecture}
\newtheorem*{question}{Open Question}
\newtheorem{claim}{Claim}

\theoremstyle{definition}
\newtheorem{definition}[theorem]{Definition}

\theoremstyle{remark}
\newtheorem*{remark}{Remark}

\begin{abstract}
We prove the Voronoi conjecture for five-dimensional parallelohedra. Namely, we show that if a convex five-dimensional polytope $P$ tiles $\R^5$ with translations, then $P$ is an affine image of the Dirichlet-Voronoi polytope for a five-dimensional lattice.

Our proof is based on an exhaustive combinatorial analysis of possible dual 3-cells and incident dual 4-cells encoding local structures around two-dimensional faces of five-dimensional parallelohedron $P$ and their edges aiming to prove existence of a free direction for $P$ paired with new properties established for parallelohedra (in any dimension) that have a free direction that guarantee the Voronoi conjecture for $P$.
\end{abstract}

\maketitle

\tableofcontents

\section{Introduction}\label{sec:intro}

Tilings serve as source of numerous patterns for art objects as well as inspiration for mathematical notions such as symmetry groups associated with the whole tiling or with an underlying point set. Particularly, Hilbert's eighteenth problem \cite{Hil} dedicated to geometry and symmetries of lattices asks about finiteness of number of classes of space groups (discrete groups of isometries with compact fundamental region) and about existence of polytope that tiles the space with congruent copies without being a fundamental region of any space group, and about the densest (sphere) packing in the three-dimensional space. All parts of Hilbert's eighteenth problem have been solved. Bieberbach \cite{Bie1,Bie2} proved the finiteness of the family of space groups in $\R^d$ and Reinhardt \cite{Rei} constructed an example of a non-convex polytope in $\R^3$ that answers that part of Hilbert's eighteenth problem. The last part of the problems was resolved by Hales \cite{Hal,Hal2} who proved the Kepler conjecture on the densest sphere packing in $\R^3$.

However, even for small dimensions, it is complicated to give a complete classification of convex polytopes that can tile Euclidean space with congruent copies with or without restricting to fundamental regions of space groups. 

Without an attempt to make a complete survey, we mention a few results regarding such polygons and polytopes. For $\mathbb{R}^2$, a proof of completeness of the list of known convex pentagons that tile the plane was announced only recently by Rao \cite{Rao}. For three-dimensional space, even the maximal number of faces of a polytope that tiles $\R^3$ with congruent copies is unknown; the best known example has 38 faces and is due to Engel \cite{Eng38}, see also \cite{GC}. Also, there are polytopes in $\R^3$, for example the Schmitt-Conway-Danzer polytope \cite[Sect. 7.2]{Sen}, that tile the space only in aperiodic way assuming that each two tiles are congruent using a rigid motion symmetry while reflections are not allowed. When the tiles are not required to be polytopes or to be convex, even one tile may enforce aperiodic structure considering all isometries as in the case of the Socolar-Taylor tile \cite{ST} in $\R^2$ or the Hat or Spectre tiles by Smith, Myers, Kaplan, and Goodman-Strauss \cite{hat,spectre}, or translations only as was recently shown by Greenfeld and Tao \cite{GT} in the space of sufficiently large dimension. We refer to \cite{Hand} and references therein for a more comprehensive survey of the topic.

In this paper we restrict our attention to tilings of Euclidean space with convex polytopes where every two tiles are translation of each other, the tilings with {\it parallelohedra}. The systematic study of parallelohedra and their properties goes back to Fedorov \cite{Fed}, Minkowski \cite{Min}, Voronoi \cite{Vor}, and Delone (Delaunay)\footnote{Boris Delone (Delaunay), a Russian and Soviet mathematician of French descent. He used the French spelling Delaunay in earlier works and the transliteration of Russian spelling Delone in later works. We use the latter spelling in the paper but the actual spelling in the cited references might be different.} \cite{Del}. The seminal work of Voronoi \cite{Vor} together with another seminal work on geometry of numbers by Minkowski \cite{Min10} introduced parallelohedra to the study of quadratic forms from the geometric point of view leading to advances in questions of lattice packings, coverings, and numerous applications of lattices and associated reduction theories.

The study of parallelohedra and their properties is inherently connected with the study of lattices. A {\it lattice} in $\R^d$ is (a translation) of the set of all integer linear combinations of some basis of $\R^d$. For a fixed lattice $\Lambda$, its {\it Dirichlet-Voronoi polytope}, or just {\it Voronoi polytope} is the set of of points that are closer to a fixed point $x\in \Lambda$ than to any other point of $\Lambda$. It is obvious that the Voronoi polytopes of lattices are parallelohedra.

The geometric properties of lattices captured by the Dirichlet-Voronoi construction and its dual Delone construction \cite{Del34} are related to various questions on lattice packings and coverings for spheres as well as for other convex bodies; we refer to work of Sch\"urmann and Vallentin \cite{SV} on computational approaches to lattice sphere packings and coverings, review of Gruber \cite{Gru} on lattice packings and coverings with convex bodies, and breakthrough works of Viazovska \cite{Via} and Cohn, Kumar, Miller, Radchenko and Viazovska \cite{CKMRV} on densest sphere packings in dimensions 8 and 24 for more details and additional references.

General parallelohedra appear in the study of spectral sets in $d$-dimensional space. As it was recently shown by Lev and Matolcsi \cite{LM},  a convex body $\Omega\subset \R^d$ is a spectral set, i.e. if there is an orthogonal basis of exponential functions in $L^2(\Omega)$, if and only if $\Omega$ is a parallelohedron as it was conjectured by Fuglede \cite{Fug}. For non-convex sets the Fuglede conjecture was disproved by Tao \cite{Tao} in dimensions $d\geq 5$. We refer to \cite{LM} and references therein for more details on the status of both directions of the Fuglede conjecture for general sets.

One of the most intriguing and still open conjectures in parallelohedra theory is the Voronoi conjecture \cite{Vor} that connects the family of all $d$-dimensional parallelohedra with $d$-dimensional lattices and their Dirichlet-Voronoi cells. This conjecture can be formulated in rather simple terms and goes back to Voronoi's study of geometric theory of positive definite quadratic forms \cite{Vor}.

\begin{conjecture}[G.~Voronoi]
For every $d$-dimensional parallelohedron $P$ there exists a $d$-dimensional lattice $\Lambda$ and an affine transformation $\mathcal A$ such that $\mathcal A(P)$ is the Dirichlet-Voronoi polytope of $\Lambda$.
{\sloppy

}
\end{conjecture}

Thus, the Voronoi conjecture claims that every convex polytope that can be considered as a fundamental region of a lattice as subgroup of $\R^d$, can be obtained by the Dirichlet-Voronoi construction, possibly followed by an affine transformation.

It is worth noting that in the works on packings and coverings mostly Dirichlet-Voronoi parallelohedra appear while the Fuglede conjecture and corresponding results concern general parallelohedra. The Voronoi conjecture essentially claims that every parallelohedron is a Dirichlet-Voronoi parallelohedron.

The Voronoi conjecture is proved in small dimensions $d\leq 4$. Two- and three-dimensional cases are usually treated as folklore and likely can be attributed to Voronoi as the lists of parallelohedra in these dimensions were known by the time the conjecture was formulated. For $d=2$, only parallelograms and centrally symmetric hexagons are parallelohedra, and three-dimensional parallelohedra were obtained by Fedorov \cite{Fed}, see Figure \ref{pict:fedorov}. Delone \cite{Del} proved the Voronoi conjecture in $\R^4$ while also providing a list of 51 four-dimensional parallelohedra which was completed by Stogrin \cite{Sto} who found the last 52nd four-dimensional parallelohedron.

\begin{center}
\begin{figure}[!ht]
\includegraphics[width=0.6\textwidth]{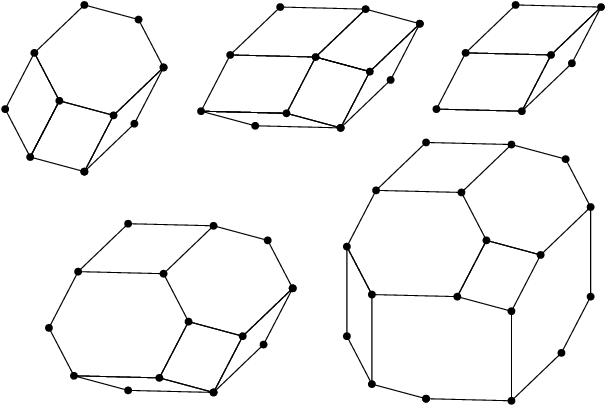}
\caption{Five three-dimensional parallelohedra: hexagonal prism, rhombic dodecahedron, parallelepiped, elongated dodecahedron, and truncated octahedron.}
\label{pict:fedorov}
\end{figure}
\end{center}

Another series of results on the Voronoi conjecture involves restrictions on local structure of face-to-face tilings by parallelohedra; this requirement originates from the classical approach to parallelohedra by Fedorov. Various types of combinatorial restrictions on the local structure around face of $P$ imply that $P$ satisfies Voronoi conjecture as shown by Voronoi \cite{Vor}, Zhitomirski \cite{Zhi}, and Ordine \cite{Ord}; we give more details on these results in Section \ref{sec:def}. We also would like to mention recent results of the author with Gavrilyuk and Magazinov \cite{GGM} and of Grishukhin \cite{Gri-cells} that prove the Voronoi conjecture for parallelohedra with global combinatorial properties.

Erdahl \cite{Erd} proved the Voronoi conjecture for parallelohedra that are zonotopes. This can be reformulated in terms of regularity for oriented matroids, see \cite{DG}, for example. We also refer to \cite[Section 3.2]{Hand} as another source of known results on the Voronoi conjecture.

The main result of this paper is a proof of the Voronoi conjecture for five-dimensional parallelohedra, Theorem \ref{thm:main}. This theorem also implies that the list of 110~244 five-dimensional Voronoi parallelohedra obtained in \cite{5dim}, i.e. the list of combinatorial classes of Voronoi polytopes of five-dimensional lattices, is the complete list of combinatorial types of parallelohedra in $\R^5$ which is summarized in Corollary \ref{cor:number}.

It should be mentioned that some sources refer to the paper of Engel \cite{Eng}, with additional computations described in \cite{Eng00}, for a proof of Voronoi conjecture in $\R^5$. The main result stated in \cite{Eng} for $\R^5$ claims that ``every parallelohedron in $\R^5$ is combinatorially equivalent to a Voronoi parallelohedron''. We have a strong doubt that this statement, and consequently the Voronoi conjecture in $\R^5$, has a rigorous justification in \cite{Eng} even if accompanied by computational results from \cite{Eng00}. In our opinion the methods used by Engel involve only operations of zone contraction and zone extension for Dirichlet-Voronoi parallelohedra of lattices represented using the cone of positive definite quadratic forms and studying faces of subcones that represent the same Delone tiling and there is no justification in \cite{Eng,Eng00} that these methods can be used to obtain all five-dimensional parallelohedra. The completeness of classification is crucial for Engel's conclusion. We give more details on our interpretations of Engel's contributions to the progress in Voronoi conjecture in papers \cite{Eng,Eng00} in Appendix~\ref{sec:appendix}.

The final judgment on the status of Engel's papers \cite{Eng,Eng00} and the results presented there is outside the scope of our work.

We would like to emphasize that our approach to the Voronoi conjecture is completely different from Engel's approach through classification of ``totally contracted parallelohedra'' that is claimed in \cite{Eng}. Instead, our approach is based on a careful analysis of local combinatorics of parallelohedral tilings and of global combinatorics of face lattices of five-dimensional parallelohedra without apriori assumption that any process gives any classification of parallelohedra in $\R^5$; see Section \ref{sec:main} for details.

The paper is organized as follows. In Section \ref{sec:def} we introduce definitions and main concepts and present key known properties of parallelohedra that are used in our proof. In Section \ref{sec:lemmas} we prove several lemmas that are crucial for our approach to five-dimensional parallelohedra. We pay special attention to combinatorics of local structure of parallelohedra tilings as this is the main tool that we use. 

In Section \ref{sec:main} we provide an outline for the proof of the Voronoi conjecture in $\R^5$ and in Sections \ref{sec:free} through \ref{sec:py-py-py} we provide all the details for the proof.

The last Section \ref{sec:final} is devoted to discussion on parallelohedra and the Voronoi conjecture in higher dimensions.

\section{Definitions and key properties}\label{sec:def}

In this section we give an overview of known properties of parallelohedra and dual cells that we need further. In most cases we state the properties for $d$-dimensional parallelohedra without restricting to the five-dimensional case.

\begin{definition}
A convex polytope $P$ in $\R^d$ is called a {\it parallelohedron} if $P$ tiles $\R^d$ with translated copies.
\end{definition}

In the classical setting, the tiling with translated copies of $P$ must be a face-to-face tiling, i.e. brickwall-like tilings are prohibited. However as it was shown later the face-to-face restriction is redundant. Particularly, a convex $d$-dimensional polytope $P$ is a parallelohedron if and only if $P$ satisfies the following {\it Minkowski-Venkov conditions}.
\begin{enumerate}
\item $P$ is centrally symmetric;
\item Each facet of $P$ is centrally symmetric;
\item Projection of $P$ along any of its face of codimension $2$ is a parallelogram or centrally symmetric hexagon.
\end{enumerate}
Minkowski \cite{Min} proved that every convex polytope that tiles $\R^d$ with translated copies in face-to-face manner satisfies first two conditions. Venkov \cite{Ven} proved that all three conditions are necessary and sufficient for a convex polytope $P$ to tile $\R^d$ with translated copies in face-to-face or non-face-to-face way; McMullen \cite{McM} (see also \cite{McM-priority}) obtained the results of Venkov independently. We also refer to the work of Groemer \cite{Gro1} for the necessity of the first two of Minkowski-Venkov conditions in some cases of packings, not necessarily face-to-face. The first two Minkowski-Venkov conditions are also necessary for coverings with constant multiplicity as shown by Gravin, Robins, and Shiryaev \cite{GRS}. We also note that the third Minkowski-Venkov condition plays a role in multi-layered tilings at least in small dimensions \cite{HSYZ}.

For a fixed parallelohedron $P$ there is a unique face-to-face tiling of $\R^d$ with translated copies of $P$ assuming one copy is centered at the origin and from now on we will consider only the case of this particular tiling. In that case the centers of the polytopes of the tiling form a $d$-dimensional lattice.

\begin{definition}\label{def:lattice}
We use the notations $\mathcal T_P$ and $\Lambda_P$ for the tiling and the lattice respectively assuming $P$ is centered at the origin. The lattice $\Lambda_P$ is called the {\it lattice associated with} $P$, or the {\it lattice of the tiling} $\mathcal T_P$.
 
The tiling $\mathcal T_P$ is preserved under translations by vectors from $\Lambda_P$ and by central symmetries in the points of $\frac 12\Lambda_P$ that preserve $\Lambda_P$.
\end{definition}

\subsection{Dual cells} \hfill 

In the course of our proof we mainly study local combinatorics of the tiling $\mathcal T_P$. The main tool we use is the dual cell technique; the method goes back to Delone \cite{Del} and the study of four-dimensional parallelohedra. The dual cell of a face $F$ of $\mathcal T_P$ encodes which copies of $P$ are incident to $F$.

\begin{definition}
Let $F$ be a non-empty face of $\mathcal T_P$. The {\it dual cell} $\mathcal D(F)$ of $F$ is the set of all centers of copies of $P$ in $\mathcal T_P$ that are incident to $F$, so $$\mathcal D(F):= \left\{ x\in\Lambda_P|F\subseteq (P+x) \right\}.$$
If $F$ is a face of codimension $k$, then we say that $\mathcal D(F)$ is a {\it dual cell of dimension $k$}, or {\it dual $k$-cell}.

If $F$ is a facet, then $\mathcal D(F)$ contains exactly two points and a segment connecting these two points is called a {\it facet vector}. Facet vectors correspond to pairs of copies of $P$ that share facets in $\mathcal T_P$.
\end{definition}

The collection of all dual cells associated with non-empty faces of $\mathcal T_P$ inherits a face lattice structure dual to the face lattice structure of the tiling $\mathcal T_P$. Namely, if a face $F$ is a subface of a face $F'$, then the cell $\mathcal D(F')$ is a subcell of the cell $\mathcal D(F)$. Hence the set of all dual cells form a cell complex that we denote $\mathcal C_P$.

In a specific case when $P$ is the Dirichlet-Voronoi cell for $\Lambda_P$, the dual cell of a face $F$ is (the vertex set of) a face of the Delone tesselation for $\Lambda_P$. Particularly, the dual cells of vertices of $\mathcal T_P$ are the Delone polytopes for $\Lambda_P$ and  these dual cells tile $\R^d$. Consequently, if the Voronoi conjecture is true for $P$, then dual cells are affine images of vertex sets of faces of Delone polytopes with inherited face lattice, so the dual cell should carry the structure of convex polytopes. In certain cases this structure can be established without prior assumption that $P$ satisfies the Voronoi conjecture.

\begin{definition}
Let $\mathcal D(F)$ be a dual $k$-cell. If the face lattice of $\mathcal D(F)$ within $\mathcal C_P$ coincides with the face lattice of the convex polytope $T:=\conv \mathcal D(F)$, then we say that $\mathcal D(F)$ is {\it combinatorially equivalent} to $T$, or just that $\mathcal D(F)$ is {\it combinatorially} $T$.

We note that this definition requires that $T$ is a $k$-dimensional polytope however this is not proved in general for every $P$ and every $k$.
\end{definition}

The theorem of Voronoi \cite{Vor} can be formulated in terms of dual $d$-cells.

\begin{theorem}[G.~Voronoi]
If all dual $d$-cells of $\mathcal T_P$ for $d$-dimensional $P$ are combinatorially $d$-simplices, then the Voronoi conjecture is true for $P$.
\end{theorem}

The Minkowski-Venkov conditions imply that there are only two types of dual $2$-cells. For a fixed face $F$ of codimension 2 of $P$, if a projection of $P$ along $F$ is a centrally symmetric hexagon, then the dual cell $\mathcal D(F)$ is combinatorially triangle, and if this projection is a parallelogram, then $\mathcal D(F)$ is combinatorially parallelogram, see Figure \ref{pict:2cells}.

\begin{center}
\begin{figure}[!ht]
\includegraphics[width=0.6\textwidth]{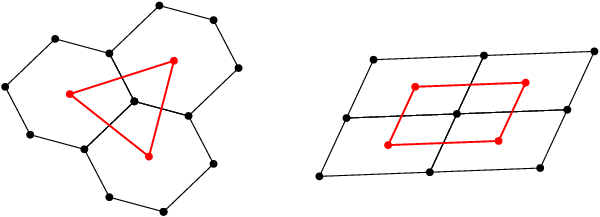}
\caption{Two types of dual 2-cells. Original parallelohedra are black polygons and dual cells are red.}
\label{pict:2cells}
\end{figure}
\end{center}

The theorem of Zhitomirski \cite{Zhi} can be stated in terms of dual cells as well.

\begin{theorem}[O.~Zhitmorski]
If all dual $2$-cells of $\mathcal T_P$ for $d$-dimensional $P$ are combinatorially triangles, then the Voronoi conjecture is true for $P$.
\end{theorem}

The complete list of dual 3-cells is also known. It was established by Delone \cite{Del} (see also \cite{Mag3cells}) as an intermediate step for his proof of the Voronoi conjecture in $\R^4$.

\begin{theorem}
If $F$ is a codimension $3$ face of $d$-dimensional parallelohedron $P$, then $\mathcal{D}(F)$ is combinatorially equivalent to one of next five $3$-dimensional polytopes.
\begin{itemize}
\item Tetrahedron;
\item Octahedron;
\item Pyramid over parallelogram;
\item Triangular prism;
\item Cube.
\end{itemize}
\end{theorem}

The first three types of dual 3-cells above exhibit a ``connectivity'' property in the following sense. Each pair of edges within one dual 3-cell of that type can be connected by a path of triangular dual 2-cells. This property was exploited by Ordine \cite{Ord} (see also \cite{Ord-arxiv}) in the following theorem.

\begin{theorem}[A.~Ordine]
If all dual $3$-cells of $\mathcal T_P$ for $d$-dimensional $P$ are combinatorially tetrahedra, octahedra, or pyramids over parallelograms, then the Voronoi conjecture is true for $P$.
\end{theorem}

The complete list of dual $k$-cells for $k>3$ is not known however we expect that this list coincides with the list of lattice Delone polytopes of dimension $k$ which are known for dimension $k\leq 6$, see \cite{Dut} for details.

\begin{definition}
A non-empty face $F$ of $\mathcal T_P$ is called a {\it contact face} if $F$ is an intersection of two copies of $P$ within $\mathcal T_P$. So for some $x,y\in \Lambda_P$ 
$$F=(P+x)\cap (P+y).$$
\end{definition}

%
%

In that case the face $F$ and its dual cell $\mathcal D(F)$ are centrally symmetric with respect to $\frac{x+y}{2}\in \frac 12\Lambda_P$ as this central symmetry preserves $\mathcal T_P$. The point $\frac{x+y}{2}$ is one of the points of the shrinked lattice $\frac12\Lambda_P$ that represents all half-lattice points. The half-lattice points are in bijection with contact faces of $\mathcal T_P$. A point $z\in \frac12\Lambda_P$ is in the relative interior of a unique face $F$ of $\mathcal T_P$. The central symmetry in $z$ preserves $\mathcal T_P$ and hence it preserves the dual cell of $F$ and $F$ itself. In that case $F$ is the intersection of any two copies of $\mathcal T_P$ centered in opposite points of $\mathcal D(F)$ as this intersection is also preserved by the central symmetry in $z$ and there is only one face of $\mathcal T_P$ that can be preserved by the central symmetry. 

Moreover, the central symmetry of dual cells is a signature property of contact faces. If $\mathcal D(F)$ is centrally symmetric, then $F$ is a contact face and centers of $F$ and $\mathcal D(F)$ coincide as it follows from the previous paragraph considering $z$ as the center of $\mathcal D(F)$. In particular, all facets of $\mathcal T_P$ are contact faces with dual cells being combinatorially segments. Among dual 2- and dual 3-cells, those combinatorially equivalent to parallelograms and to octahedra and parallelepipeds, respectively, are dual cells of contact faces while others are not.


In addition to Euclidean space $\R^d$ and the lattice $\Lambda_P$ we use two additional (finite) vector spaces. Namely, the {\it space of parity classes} $\Lambda_p:=\Lambda_P/(2\Lambda_P)$ and the {\it space of half-lattice points} $\Lambda_{1/2}:=\left(\frac 12\Lambda_P\right)/\Lambda_P$. For a point $x\in\Lambda_P$ we call the coset $x+\Lambda_P/(2\Lambda_P)$ the {\it parity class} of $x$.

As vector spaces both $\Lambda_p$ and $\Lambda_{1/2}$ are isomorphic to $\mathbb F_2^d$, a $d$-dimensional vector space over two-element field $\mathbb{F}_2$, but they serve quite different roles. The space of parity classes gives us all possible options for various points in exhaustive approaches throughout Sections \ref{sec:free} through \ref{sec:py-py-py} and the space of half-lattice points is used to extract combinatorics of dual cell complex $\mathcal{C}_P$ and contact faces in particular. 

We use notations $[x_1,x_2,\ldots,x_d]$ for elements of $\Lambda_p$ and $\langle x_1,x_2,\ldots,x_d\rangle$ for elements of $\Lambda_{1/2}$ in coordinate representation.

The following lemma is a classical result in parallelohedra theory, see \cite{DS} for example. 

\begin{lemma}\label{lem:parity}
If $F$ is a face of $\mathcal T_P$, then $\mathcal D(F)$ contains at most one representative from each parity class.
\end{lemma}
\begin{remark}
The following proof of this lemma is using a combinatorial approach that is later used to study other properties of dual cells in Section \ref{sec:lemmas} and throughout our proof of the main result.

Suppose $x$ and $y$ belong to the same parity class and $x,y\in \mathcal D(F)$. The polytopes $P+x$ and $P+y$ have non-empty intersection, so they both must contain the midpoint $\frac{x+y}{2}$ of $xy$ because the central symmetry with respect to $\frac{x+y}{2}$ preserves this intersection. However $\frac{x+y}{2}\in \Lambda_P$ and is an internal point of another copy of $P$ which is impossible.
\end{remark}

\subsection{Canonical scaling} \hfill 

One of the most used approaches to prove the Voronoi conjecture for a class of parallelohedra involves a proof of existence of canonical scaling for polytopes within that class; we also refer to \cite{DGequiv} for other conditions that imply the Voronoi conjecture for a particular family of parallelohedra. The approach was first used by Voronoi \cite{Vor} for primitive parallelohedra. We use this approach in Sections \ref{sec:free} and \ref{sec:py-py-py}.

\begin{definition}\label{def:scaling}
Let $\mathcal{T}_P^{d-1}$ be the set of all facets of $\mathcal T_P$. For every facet $F$ let $n_F$ be one of its two unit normals. Let $s: \mathcal{T}_P^{d-1}\longrightarrow \R_+$ be defined on every facet of $\mathcal T_P$. For every codimension 2 face $X$ of the tiling $\mathcal T_P$ we look at the following equality of vectors.
\begin{itemize}
\item In case $\mathcal D(X)$ is combinatorially triangle, equivalently $X$ is non-contact, then there are exactly three facets $F$, $G$, and $H$ incident to $X$. We require that for a certain choice of signs 
$$\pm s(F)n_F\pm s(G)n_G\pm s(H)n_H=0.$$ 
\item In case $\mathcal D(X)$ is combinatorially parallelogram, equivalently $X$ is contact, then there are exactly four facets $F$, $G$, $H$, and $I$ incident to $X$. We require that for a certain choice of signs 
$$\pm s(F)n_F\pm s(G)n_G\pm s(H)n_H\pm s(I)n_I=0.$$ 
\end{itemize}
If these equalities can be satisfied for $s$ for every codimension 2 face $X$ of $\mathcal T_P$, then $s$ is called a {\it canonical scaling} for $P$ (or for $\mathcal T_P$).
\end{definition}

Effectively, the first condition dictates that the values of canonical scaling on three facets with common non-contact face of codimension 2 are proportional to absolute values of coefficients of their normals in the corresponding linear dependence; this linear dependence is essentially unique for the three linearly dependent normal vectors in two-dimensional subspace orthogonal to $F\cap G\cap H$. The second condition only says that values of canonical scaling on two opposite facets at a contact face of codimension 2 are equal. However it could be strengthened to equality of canonical scaling for every pair of parallel facets.

As it was shown by Voronoi \cite{Vor}, a parallelohedron $P$ satisfies the Voronoi conjecture if and only if $\mathcal T_P$ exhibits a canonical scaling. This equivalence was used by Voronoi \cite{Vor}, Zhitomirski \cite{Zhi}, and Ordine \cite{Ord} to prove their theorems on Voronoi conjecture for respective classes of parallelohedra.

The first condition for canonical scaling can be transformed into the following notion.

\begin{definition}\label{def:gain}
Suppose $F$ is a face of codimension 2 with triangular dual cell $\mathcal D(F)=ABC$. Each edge of $ABC$ is a dual cell of a facet of $\mathcal T_P$; we denote normals of these facets as $n_{AB}$, $n_{BC}$ and $n_{CA}$. There is a unique (up to non-zero factor) linear dependence $$\alpha_{AB}n_{AB}+\alpha_{BC}n_{BC}+\alpha_{CA}n_{CA}=0$$ for these normals with non-zero coefficients $\alpha_{AB}, \alpha_{BC}$, and $\alpha_{CA}$.

For a pair of facet vectors $AB$ and $AC$ that are incident to one triangular dual cell $ABC$ we define the {\it gain function} $\gamma(AB,AC)$ as 
$$\gamma(AB,AC):=\frac{|\alpha_{AC}|}{|\alpha_{AB}|}.$$
This notion is naturally extended to a sequence of facet vectors $f_1,\ldots,f_m$ where each two consecutive facet vectors belong to one triangular dual cell as
$$\gamma(f_1,\ldots,f_m):=\gamma(f_1,f_2)\cdot\ldots\cdot \gamma(f_{m-1},f_m).$$
In other words, the gain function on a sequence of facet vectors tracks how values of a hypothetical canonical scaling would change if we sequentially enforce the conditions from Definition \ref{def:scaling} for each consecutive pair of facet vectors.
\end{definition}

As it was shown by Garber, Gavrilyuk and Magazinov in \cite{GGM}, a canonical scaling for $P$ exists if and only if the gain function $\gamma$ is 1 on every appropriate cycle within $\mathcal T_P$. We will use this property in Section \ref{sec:py-py-py}.

\subsection{Free directions} \label{sec:freedir} \hfill 

Free directions can be used to ``extend'' or ``contract'' parallelohedra while generally preserving combinatorial structure of the tilings. These directions were used by Delone \cite{Del} to establish a certain layered structure of tilings by parallelohedra.

\begin{definition}
Let $P$ be a parallelohedron and let $v$ be a non-zero vector. We say that $P$ is {\it free} in the direction of $v$ if there exists a segment $I$ parallel to $v$ such that the Minkowski sum $P + I$ is a parallelohedron of the same dimension as $P$. We say that the direction of $v$ is a {\it free direction} for $P$ as well as every non-zero segment parallel to $v$ is a {\it free direction} for $P$.
\end{definition}

\begin{remark}
If $P + I$ is a parallelohedron then the sum $P + I'$ is a parallelohedron for any segment $I'$ parallel to $I$. Indeed, the combinatorics of $P + I$ and $P + I'$ is the same and the Minkowski-Venkov conditions for these polytopes can be satisfied only simultaneously.
\end{remark}

Free directions of general parallelohedra and their relation to the Voronoi conjecture are relatively well studied, we refer to papers of Magazinov \cite{Mag}, Horv\'ath \cite{Hor}, Grishukhin \cite{Gri} and references therein.

The following criterion can be used to determine whether the direction of segment $I$ is a free direction for $P$. It was initially stated by Grishukhin \cite{Gri} but a complete proof was given only in \cite{DGM} by Dutour Sikiri\'{c}, Grishukhin and Magazinov.

\begin{lemma}\label{lem:free}
A non-zero vector $v$ spans a free direction for $P$ if and only if every triangular dual cell $xyz = \mathcal D(G)$, where $G$ is a non-contact $(d - 2)$-face of $\mathcal T_P$, satisfies the following condition. If $F(xy)$ is a $(d - 1)$-face of $\mathcal T_P$ such that $\mathcal D(F(xy)) = xy$, and a similar definition applies to $F(xz)$ and $F(yz)$, then at least one of the faces $F(xy)$, $F(xz)$ and $F(yz)$ is parallel to $v$.
\end{lemma}

The next lemma summarizes some useful combinatorial properties of a parallelohedron $P + I$ that have been established to date. We use these properties in Section \ref{sec:free}.

\begin{lemma}\label{lem:p+i}
Let $P$ be a $d$-dimensional parallelohedron with a free direction $I$. If $P+I$ satisfies the Voronoi conjecture, then $P$ satisfies the Voronoi conjecture.
\end{lemma}
\begin{proof}
See~\cite[Theorem~4]{Gri2} or~\cite{Vegh-contraction}.
\end{proof}

\begin{remark}
The proof by Grishukhin~\cite{Gri2} relies on the technique of canonical scaling, while V\'egh~\cite{Vegh-contraction} provided an explicit construction of the affine transformation from $P$ to a Dirichlet-Voronoi polytope given a transformation for $P + I$.
\end{remark}

\section{New lemmas}\label{sec:lemmas}

In this section we prove several new lemmas that we use in the proof of our main result. 

First of all we formulate several properties of dual cells that are crucial for our approach to five-dimensional parallelohedra. However, it is worth noting that all lemmas in this section hold in any dimension.

\begin{lemma}\label{lem:intersection}
Let $F$ and $G$ be two faces of $P$ and let $H$ be the minimal face of $P$ that contains both $F$ and $G$. Then $$\mathcal D(H)=\mathcal D(F)\cap\mathcal D(G).$$
\end{lemma}
\begin{proof}
Let $Q$ be the copy of $P$ centered at a point of $\Lambda_P$. The polytope $Q$ contains $F$ and $G$ if and only if $P\cap Q$ contains $F$ and $G$. The intersection $P\cap Q$ is a face of $P$ and it contains $F$ and $G$ if and only if it contains $H$. Hence $Q$ contains $F$ and $G$ if and only if $Q$ contains $H$. This implies the equality for dual cells.
\end{proof}

\begin{definition}
Let $D$ be a dual cell. We define the {\it set of midpoints} for the dual cell $D$ as the set of all classes of midpoints within $D$, so
$$M_D:=\left\{\left.\frac{X+Y}{2}+\Lambda_P\right| X,Y\in D\right\}\subseteq \Lambda_{1/2}.$$
Here $\frac{X+Y}{2}$ is the midpoint of the segment $XY$. Note, that we do not require $X$ and $Y$ to be different, so the class $\langle 0,0,\ldots,0\rangle$ is always in $M_D$.
\end{definition}

Next two lemmas transform translation invariance of $\mathcal T_P$ into invariance of dual cells. Particularly, they use that if one representative of $\Lambda_{1/2}$ class is the center of a dual $k$-cell, then all points from that class are centers of translations of this $k$-cell. We also refer to \cite[Lem. 6]{Ord} for a very similar result that can be used to prove these lemmas.

\begin{lemma}\label{lem:translated_cell}
Let $D$ be a dual cell and let $F$ be a contact face of $P$ with the center $c_F$. Let $x$ be the midpoint of a segment connecting two points of $D$. If $x$ and $c_F$ represent the same class in $\Lambda_{1/2}$ then $D$ contains the translated copy $\mathcal D(F)+\overrightarrow{c_Fx}$ of the dual cell of $F$.
\end{lemma}
\begin{proof}
Let $y$ and $z$ be points in $D$ such that $x=\frac{y+z}{2}$. Two polytopes $P+y$ and $P+z$ have a non-empty intersection, so their intersection is a contact face $G$ of $\mathcal T_P$ with center $x$ such that $\mathcal{D}(G)$ is a subcell of $D$ because $(P+y)\cap (P+z)$ contains the face corresponding to $D$.

The translation by vector $\overrightarrow{c_Fx}$ moves $c_F$ to $x$ and therefore moves the contact face $F$ centered at $c_F$ into the contact face $G$ centered at $x$. Thus the translation of the dual cell $\mathcal D(F)$ is $\mathcal D(G)$ which is contained in $D$.
\end{proof}

\begin{lemma}\label{lem:translated_pair}
Let $D$ be a dual cell and let $A$ and $B$ be two points in one dual cell of $P$. Let $x$ be the midpoint of a segment connecting two points of $D$. If $x$ and the midpoint $c$ of $AB$ represent the same class in $\Lambda_{1/2}$ then $D$ contains the translated copy $AB+\overrightarrow{cx}$ of the segment $AB$.
\end{lemma}
\begin{proof}
We use Lemma \ref{lem:translated_cell} for the cell $D$ and the face $F$ which is the intersection of the copies of $P$ centered at $A$ and $B$.
\end{proof}

Mostly we will use this lemma when $F$ is a facet or, which is the same, when $\mathcal D(F)$ is the segment $AB$ as in the two lemmas below.

\begin{lemma}\label{lem:midpoints}
Let $KL$ be a facet vector. If $M$ and $N$ are two points within one dual cell such that the midpoints of $KL$ and $MN$ belong to the same class in $\Lambda_{1/2}$, then $\overrightarrow{KL}=\pm\overrightarrow{MN}$.
\end{lemma}
\begin{proof}
Two copies of $P$ centered at $M$ and $N$ have a non-empty intersection $F$. We use the previous Lemma \ref{lem:translated_cell} for the dual cell $KL$ and points $M$ and $N$ that both belong to $\mathcal D(F)$. The translation of $\mathcal D(F)$ must fit within $KL$ which implies that $\mathcal D(F)$ contains exactly two points $M$ and $N$ and segments $MN$ and $KL$ are translations of each other. Hence $\overrightarrow{KL}=\pm\overrightarrow{MN}$.
\end{proof}

\begin{lemma}\label{lem:3+1}
Let $D$ be a dual cell. Suppose $K$, $L$ and $M$ are three points of $D$ such that segments $KL$, $LM$, and $MK$ are facet vectors. Then $D$ does not contain a point from the parity class of $K+L+M$.
\end{lemma}
\begin{proof}
Suppose $D$ contains a point $N = K+L+M \pmod {2\Lambda_P}$. The midpoints of $KL$ and $MN$ differ by a vector of $\Lambda_P$ because $\frac{K+L}{2}=\frac{M+N}{2} \pmod {\Lambda_P}$, hence Lemma \ref{lem:midpoints} for the facet vector $KL$ and pair of points $M$ and $N$ within $D$ implies that $\overrightarrow{KL}=\pm\overrightarrow{MN}$.

Similarly $\overrightarrow{KM}=\pm\overrightarrow{LN}$ and $\overrightarrow{LM}=\pm\overrightarrow{KN}$ but all these three equalities cannot be satisfied simultaneously.
\end{proof}

Also we will use the following corollary of the criterion from Lemma \ref{lem:free} stated in terms of the set of midpoints of dual cell of an edge.

\begin{lemma}\label{lem:free_space}
Let $I$ be an edge of $P$. If there is a $(d-1)$-dimensional linear subspace $\pi$ of $\Lambda_{1/2}$ such that each class of $\pi$ is in $M_{\mathcal D(I)}$ or corresponds to a non-facet contact face of $P$, then $I$ is a free direction of $P$.
\end{lemma}
\begin{proof}
Let $KLM$ be any triangular dual cell of $\mathcal T_P$. Points $K$, $L$, and $M$ belong to different parity classes so the midpoints $\frac{K+L}{2}$, $\frac{L+M}{2}$, and $\frac{M+K}{2}$ represent different classes in $\Lambda_{1/2}$. The sum
$$\frac{K+L}{2}+\frac{L+M}{2}+\frac{M+K}{2}=0\in \Lambda_{1/2},$$
hence the three midpoints together with the origin fill a two-dimensional linear subspace of $\Lambda_{1/2}$. This two-dimensional subspace has a non-trivial intersection with $\pi$, so we can assume that $\frac{K+L}{2}\in \pi$. 

The midpoint $\frac{K+L}{2}$ represents a facet, thus it coincides with the class of some midpoint of the dual cell of $I$. Lemma \ref{lem:translated_cell} implies that $\mathcal D(I)$ contains a translated copy of the edge $KL$ which means that a translation of the facet corresponding to $KL$ contains $I$. Therefore the facet corresponding to $KL$ is parallel to $I$.

Now Lemma \ref{lem:free} implies that $I$ is a free direction for $P$.
\end{proof}

\section{Main theorem and the core of the proof}\label{sec:main}

In this section we provide an outline for the proof of our main Theorem \ref{thm:main}. In the following sections we fill in all the details for each specific step of the proof. We also note that an implementation of similar approach in $\R^4$ is described in \cite{GMrms}.

\begin{theorem}\label{thm:main}
The Voronoi conjecture is true in $\R^5$.
\end{theorem}
\begin{proof}
This proof relies on several supplementary results that are proved in subsequent sections. However, whenever a proof of some implication is deferred, we give a reference to a particular section.
 
By Lemma \ref{lem:5-free}, the main result of Section \ref{sec:free}, a five-dimensional parallelohedron $P$ satisfies the Voronoi conjecture if it has a free direction. Consequently, it will be sufficient to prove that every five-dimensional parallelohedron $P$ satisfies at least one of the following properties:

\begin{enumerate}
\item $P$ has a free direction;
\item $P$ admits a canonical scaling.
\end{enumerate}

Let $P$ be a five-dimensional paralleloheron. We turn our attention to dual 3-cells associated with two-dimensional faces of $P$. By a result of Ordine~\cite{Ord} (see also \cite{Ord-arxiv}), if all dual 3-cells of $P$ are either tetrahedra, octahedra, or pyramids, then $P$ admits a canonical scaling and therefore the Voronoi conjecture is true for $P$.

According to Corollary \ref{cor:cube}, the main result of Section~\ref{sec:cube}, if $P$ has a dual 3-cell combinatorially equivalent to a cube, then $P$ has a free direction. In this case the Voronoi conjecture is true for $P$.

To this end, the situation that is still to be considered is as follows: at least one dual 3-cell for $P$ is a triangular prism, while every other dual 3-cell is a tetrahedron, a pyramid, an octahedron, or a prism.

Let $F$ be a 2-dimensional face of $\mathcal T_P$ whose dual cell $\mathcal D(F)$ is a triangular prism. By Lemmas \ref{lem:triangle} and \ref{lem:edge_cell}, two main results of Section~\ref{sec:prism}, $P$ has a free direction unless $F$ is a triangle, which we denote by $xyz$, and unless each of the dual 4-cells $\mathcal D(xy)$, $\mathcal D(xz)$ and $\mathcal D(yz)$ is (combinatorially) either a pyramid over $\mathcal D(F)$ or a prism over a tetrahedron. Let $pr(F)$ denote the number of prismatic 4-cells among the dual cells $\mathcal D(xy)$, $\mathcal D(xz)$ and $\mathcal D(yz)$. We proceed by the case analysis.

\noindent {\bf Case 1} or {\bf Prism-Prism-Prism case}. There exists $F$ with $pr(F) = 3$. According to Lemma \ref{lem:pr-pr-pr}, the main result of Section~\ref{sec:pr-pr-pr}, this is only possible if $P$ is a direct sum of parallelohedra of smaller dimensions. Hence, in particular, $P$ has a free direction and therefore satisfies the Voronoi conjecture.

\noindent {\bf Case 2}  or {\bf Prism-Prism-Pyramid case}. There exists $F$ with $pr(F) = 2$. By Lemma \ref{lem:pr-pr-py}, the main result of Section \ref{sec:pr-pr-py}, $P$ has a free direction and therefore satisfies the Voronoi conjecture. 

\noindent {\bf Case 3}   or {\bf Prism-Pyramid-Pyramid case}. There exists $F$ with $pr(F) = 1$. By Lemma \ref{lem:pr-py-py}, the main result of Section \ref{sec:pr-py-py}, at least one of the three sides of $F$ gives a free direction for $P$. Therefore $P$ satisfies the Voronoi conjecture.

\noindent {\bf Case 4}   or {\bf Pyramid-Pyramid-Pyramid case}. For every triangular face $F \subset P$ whose dual cell $\mathcal D(F)$ is a triangular prism it holds that $pr(F) = 0$. Then, by Lemma \ref{cor:py-py-py}, the main result of Section~\ref{sec:py-py-py}, $P$ necessarily admits a canonical scaling or has a free direction. In both cases $P$ satisfies the Voronoi conjecture.

The proof is now finished, since all possible cases are considered.
\end{proof}

Among all classification results that immediately follow, we mention here one particular corollary of Theorem~\ref{thm:main} that the list of Dirichlet-Voronoi parallelohedra from \cite{5dim}, i.e. the list of all combinatorial classes of polytopes that appear as Voronoi polytopes of five-dimensional lattices, is the complete list of combinatorial types of five-dimensional parallelohedra.

\begin{corollary}\label{cor:number}
There are exactly 110~244 combinatorial types of parallelohedra in $\R^5$.
\end{corollary}

\section{Parallelohedra with free direction}\label{sec:free}

In this section we prove that a parallelohedron in $\R^5$ with a free direction satisfies the Voronoi conjecture. Before proving that specific result for the five-dimensional case, we prove a general statement for Voronoi parallelohedra with free directions.

Suppose $P$ is a $d$-dimensional parallelohedron that has at least one free direction. That is, suppose there is a (non-zero) segment $I$ such that both $P$ and $P+I$ are parallelohedra. In that case the projection of $P$ along $I$ is a $(d-1)$-dimensional parallelohedron due to result of Venkov \cite{Ven-width}.

Additionally we need the following notion of a (strong) equivalence for parallelohedra, see \cite[Def. 1.1]{DGM-5dim} and discussion therein.

\begin{definition}\label{def:equiv}
Let $P$ and $P'$ be two $d$-dimensional parallelohedra. We say that $P$ and $P'$ are {\it equivalent}, if there is a combinatorial equivalence $\mathfrak{F}$ between the tilings $\mathcal T_P$ and $\mathcal T_{P'}$ that 
induces a linear isomorphism of $\Lambda_P$ to $\Lambda_{P'}$ restricting $\mathfrak{F}$ to copies of $P$ and $P'$ and then passing that restriction to the centers of the tiles in $\mathcal T_P$ and $\mathcal T_P'$.
\end{definition}

In other words, in addition to combinatorial equivalence of $P$ and $P'$ that can be obtained by restricting $\mathfrak F$ to single full-dimensional tiles of $\mathcal T_P$ and $\mathcal T_{P'}$, we require that the group action of the lattice $\Lambda_P$ (as finitely generated abelian group) on $\mathcal T_P$ transfers to the group action of the lattice $\Lambda_{P'}$ on $\mathcal T_{P'}$. A simplest example for such an equivalence is when $P$ and $P'$ are affinely equivalent. Assuming $P$ and $P'$ are centered at the origin, any affine transformation $\mathcal A$ of $P$ to $P'$ induces a combinatorial equivalence of the corresponding tilings and a linear isomorphism of the corresponding lattices. However, already in $\R^2$, all centrally symmetric hexagons are equivalent in the sense of Definition \ref{def:equiv} but give infinitely many affine classes.

On the other hand, the equivalence as defined above is potentially stronger than combinatorial equivalence of $P$ and $P'$. That is, some combinatorially equivalent parallelohedra could be non-equivalent under Definition \ref{def:equiv}. However, we are unaware of an example of two combinatorially equivalent parallelohedra that are not equivalent in the sense of Definition \ref{def:equiv} and for $d\leq 4$ any combinatorial equivalence between $P$ and $P'$ implies the stronger equivalence in the sense of Definition \ref{def:equiv}.

\begin{theorem}\label{thm:free_voronoi}
If a $d$-dimensional parallelohedron $P$ has a free direction $I$ and the projection of $P$ along $I$ satisfies the Voronoi conjecture, then $P+I$ is equivalent (in the sense of Definition \ref{def:equiv}) to the Voronoi parallelohedron for some $d$-dimensional lattice.
\end{theorem}
\begin{proof}
Let $\mathcal F(I)$ be the set of all facet vectors of $P+I$ with corresponding facets parallel to $I$. According to the result of Horv\'ath \cite{Hor}, the set $\mathcal F(I)$ spans, using integer coefficients, a $(d-1)$-dimensional sublattice $\Lambda_I$ of $\Lambda_{P+I}$. The sublattice $\Lambda_I$  coincides with the intersection $(\lin \Lambda_I) \cap \Lambda_{P+I}$ due to \cite[Lemma 3.3]{Mag} hence $\Lambda_{P+I}$ splits into layers $$\Lambda_{P+I}=\bigsqcup\limits_{n\in\Z}\Lambda_I^n$$ where $\Lambda_I^n=nx+\Lambda_I$ for some fixed $x\in\Lambda_{P+I}$. Also, if two copies of $P+I$ have a non-empty intersection, then their centers belong to the same or consecutive layers due to \cite[Lemma~3.2]{Mag}.

Let $Q$ be the projection of $P+I$ on $\lin \Lambda_I$ along $I$. We apply an affine transformation $\mathcal A$ with invariant subspace $\lin \Lambda_I$ that makes $I$ orthogonal to $\lin \Lambda_I$ and transforms $Q$ into the Dirichlet-Voronoi cell of $\mathcal A(\Lambda_I)$. Such a transformation exists because $Q$ satisfies the Voronoi conjecture by the conditions of this theorem. This transformation does not change the combinatorial types of $P+I$ and $\mathcal T_{P+I}$, or the equivalence class according to Definition \ref{def:equiv}.

First, we notice that for any $\lambda >0$ the polytope $P+\lambda I$ is a parallelohedron and is equivalent to $P+I$ in the sense of Definition \ref{def:equiv}, so we may assume that $I$ is long enough so the affine space $\lin \mathcal A(\Lambda_I^0)$ is tiled by the copies of $\mathcal A(P+I)$ centered at $\mathcal A(\Lambda_I^0)$ and long enough that in the Voronoi tiling of $\mathcal A(\Lambda_{P+I})$ (centers of) polytopes with non-empty intersection belong to the same or to adjacent layers $\mathcal A(\Lambda_I^n)$.  We claim that in this case, the parallelohedron $\mathcal A(P+I)$ is equivalent to the Dirichlet-Voronoi cell $DV_{P+I}$ of the lattice $\mathcal A(\Lambda_{P+I})$. We also can assume that $P+I$ is centered at the origin.

Let $F$ be an $m$-dimensional face of $\mathcal A(P+I)$. If $F$ belongs to copies of $\mathcal A(P+I)$ within only one layer, then the the corresponding copies in the Voronoi tiling of $\Lambda_{P+I}$ intersect by an $m$-dimensional face as well. Thus, it remains to consider the case that $F$ is an intersection of two sets of copies of $\mathcal A(P+I)$ centered in two consecutive layers; without loss of generality we can assume that these layers are $\mathcal A(\Lambda_I^0)$ and $\mathcal A(\Lambda_I^1)$. Let $x_1,\ldots, x_r\in \mathcal A(\Lambda_I^0)$ be the centers in the 0th layer and $y_1,\ldots, y_t\in \mathcal A(\Lambda_I^1)$ be the centers in the 1st layer; here $r,t\geq 1$.

Copies of $\mathcal A(Q)$ centered at $\mathcal A(\Lambda_I^0)$ give the Voronoi tiling of $\mathcal A(\Lambda_I^0)$. Therefore, the copies of $\mathcal A(Q)$ centered at $x_1,\ldots,x_r$ intersect at a face $F_I^0$ of this Voronoi tiling. Moreover, all copies incident to $F_I^0$ in the Voronoi tiling of $\mathcal A(\Lambda_I^0)$  have centers among $x_1,\ldots,x_r$ as these points of $\mathcal A(\Lambda_I^0)$ are closest to every point of $F_I^0$, Similarly, $F_I^1$ is a face of the Voronoi tiling of $\mathcal A(\Lambda_I^1)$ given by the intersection of copies of $\mathcal A(Q)$ centered at $y_1,\ldots,y_t$.

In the Voronoi tiling of $\mathcal A(\Lambda_{P+I})$, the copies of $DV_{P+I}$ centered at $x_1,\ldots, x_r$ intersect at a face $F^0$ that is projected onto $F_I^0$ along $I$; also $F_I^0$ is a subset of $F^0$. Similarly, the copies centered at $y_1,\ldots, y_t$ intersect at a face $F^1$ that is projected onto $F_I^1$ along $I$. The faces $F^0$ and $F^1$ must intersect between the two layers as $F^0+I\cdot \R$ and $F^1+I\cdot \R$ both contain $F$ and no other polytope of the Voronoi tiling of $\mathcal A(\Lambda_{P+I})$ can reach the intersection of $F^0+I\cdot \R$ and $F^1+I\cdot \R$ between 0th and 1st layers. 

We have constructed a bijection between two face lattices of tilings with copies of $\mathcal A(P+I)$ and $DV_{P+I}$ defined by finite subsets of centers of parallelohedra that induce non-empty intersections. Since these face lattices correspond to cell complexes of the same dimension and the bijection respects incidence of faces, it preserves the dimensions of faces as well. Thus, the faces $F^0$ and $F^I$ above intersect by an $m$-dimensional face of the Voronoi tiling of $\mathcal A(\Lambda_{P+I})$

If we propagate this bijection to all faces of the tiling $\mathcal T_{\mathcal A(P+I)}$ we get that $\mathcal A(P+I)$ is equivalent to the Voronoi cell of $\mathcal A(\Lambda_{P+I})$ in the sense of Definition \ref{def:equiv} as the induced bijection of the lattices is the identity isomorphism.
\end{proof}

Combining the previous theorem with Lemma \ref{lem:p+i}, results of Delone \cite{Del} on 4-dimensional parallelohedra, and results of the author with Dutour Sikiri\'c and Magazinov on five-dimensional combinatorially Voronoi parallelohedra \cite[Thm 1.3]{DGM-5dim}, we get the main result of this section.
{\sloppy

}

\begin{lemma}\label{lem:5-free}
If a five-dimensional parallelohedron $P$ has a free direction then $P$ satisfies the Voronoi conjecture.
\end{lemma}
\begin{proof}
Let $I$ be a segment of a free direction for $P$ so $P+I$ is a parallelohedron. The projection of $P$ along $I$ is a four-dimensional parallelohedron that satisfies the Voronoi conjecture according to \cite{Del}. Thus, $P+I$ is equivalent to a Voronoi parallelohedron for some five-dimensional lattice due to Theorem \ref{thm:free_voronoi}. 

In \cite[Thm. 1.3]{DGM-5dim}, the author with Dutour Sikiri\'c and Magazinov proved that if a five-dimensional parallelohedron is equivalent in the sense of Definition \ref{def:equiv} to the Voronoi polytope of a five-dimensional lattice, then it satisfies the Voronoi conjecture. Therefore, $P+I$ satisfies the Voronoi conjecture. Thus $P$ satisfies the Voronoi conjecture due to Lemma \ref{lem:p+i}.
\end{proof}

\section{Parallelohedra with cubical dual 3-cells}\label{sec:cube} 

Let $P$ be any five-dimensional parallelohedron. In this section we prove that if a $P$ has a dual 3-cell equivalent to a three-dimensional cube, then $P$ has a free direction. In this and further sections we assume that $\Lambda_P=\Z^5$ as this can be achieved using an affine transformation. Recall that in that setting $\Z^5_p$ is the vector space of parity classes and $\Z^5_{1/2}$ is the vector space of half-lattice (or half-integer) points.

\begin{lemma}\label{lem:cube}
If $F$ is a two-dimensional face of $P$ with dual cell $\mathcal{D}(F)$ equivalent to a cube, then every edge of $F$ is a free direction for $P$.
\end{lemma}
\begin{proof}
Let $e$ be an edge of $F$. Then $\mathcal D(e)$ contains $\mathcal D(F)$ and these two dual cells do not coincide. Let $A$ be any point in $\mathcal D(e)\setminus \mathcal D(F)$.

The points of $\mathcal D(F)$ represent 8 different parity classes within a three-dimensional affine subspace of $\R^5$. Therefore $\mathcal D(F)$ is a three-dimensional affine subspace of $\Z^5_p$ and $M_{\mathcal D(F)}$ is a three-dimensional linear subspace $\pi$ of $\Z^5_{1/2}$.

The parity class of $A$ differs from the parity classes of points of $\mathcal D(F)$ because these points are in one dual cell $\mathcal D(e)$. Therefore the set of 8 midpoints
$$
\pi':=\left\{\left.
\frac{A+X}{2}\right|X\in \mathcal D(F)
\right\}\subset \Z^5_{1/2}$$
is a translation of $\pi$ that differs from $\pi$.

The union $\pi\cup \pi'$ is a four-dimensional linear subspace of $\Z^5_{1/2}$ and $\pi\cup \pi' \subseteq M_{\mathcal D(e)}$. Now Lemma \ref{lem:free_space} for the edge $e$ and subspace $\pi\cup\pi'$ implies that $e$ is a free direction for $P$.
\end{proof}

\begin{corollary}\label{cor:cube}
If a five-dimensional parallelohedron $P$ has a dual 3-cell equivalent to a cube, then $P$ satisfies the Voronoi conjecture.
\end{corollary}

\section{Parallelohedra with prismatic dual 3-cells and their properties}\label{sec:prism}

In this section we prove that if a five-dimensional parallelohedron $P$ has the dual 3-cell of a face $F$ equivalent to triangular prism, then $P$ has a free direction or $F$ is a triangle. Moreover we show that the dual cells of edges of $F$ are equivalent (not only as cell complexes but as geometrical vertex sets with inherited face structure) to prisms over tetrahedra or to pyramids over triangular prisms unless $P$ has a free direction.

Suppose $\mathcal D(F)=XYZX'Y'Z'$ where $XYZ$ and $X'Y'Z'$ are the bases of the prism, and $\overrightarrow{XX'}=\overrightarrow{YY'}=\overrightarrow{ZZ'}$. We note that the three-dimensional affine subspace of $\Z^5_p$ spanned by $\mathcal D(F)$ contains parity classes of $X,Y,Z,X',Y',Z',X+Y+Z$, and $X'+Y'+Z'$ and hence dual cells of edges and vertices of $F$ contain only the prism $XYZX'Y'Z'$ in this affine span due to Lemmas \ref{lem:parity} and \ref{lem:3+1}.

\begin{lemma}\label{lem:triangle}
The parallelohedron $P$ has a free direction or $F$ is a triangle.
\end{lemma}
\begin{proof}
Suppose $F$ is not a triangle so $F$ is an $n$-gon for $n\geq 4$. For every edge $e_i, 1\leq i\leq n$ of $F$, the dual cell $\mathcal D(e_i)$ contains an additional vertex $A_i$ in a parity class outside of the three-dimensional affine subspace $\pi_F$ of $\Z^5_p$ spanned by $XYZX'Y'Z'$. The space $\Z^5_p$ is split into four three-dimensional affine planes parallel to $\pi_F$ including $\pi_F$ itself. Since $n\geq 4$ and $A_i\notin \pi_F$, at least two points, say $A_i$ and $A_j$ corresponding to edges $e_i$ and $e_j$ belong to the same translation of $\pi_F$.

Without loss of generality we can assume that the points represent the following parity classes in $\Z^5_p$
\begin{center}
\begin{tabular}{ll}
$X\in[0,0,0,0,0]$,\qquad \qquad & $X'\in[0,0,0,0,1]$,\\
$Y\in[1,0,0,0,0]$,\qquad \qquad & $Y'\in[1,0,0,0,1]$,\\
$Z\in[0,1,0,0,0]$,\qquad \qquad & $Z'\in[0,1,0,0,1]$,\\
\end{tabular}
\end{center}
and $A_i\in[0,0,1,0,0]$. The plane $\pi_F$ is given by the equation $x_3=x_4=0$ and the translation of $\pi_F$ that contains parity classes of $A_i$ and $A_j$ is given by $\pi_F':=\{x_3=1, x_4=0\}$, so $A_j\in[*,*,1,0,*]$ where each $*$ can be $0$ or $1$ independently.

There are three cases for the intersection of $\mathcal D(e_i)$ with the affine space $\pi_F'$ in $\Z^5_p$.

{\bf Case \ref{lem:triangle}.1:} The intersection $\mathcal D(e_i)\cap \pi_F'$ contains two points that differ in a coordinate other than $x_5$. One of these points is $A_i\in[0,0,1,0,0]$; we denote the second one as $A_i'$ and $A_i'$ belongs to the parity class of the form $[1,0,1,0,*]$, $[0,1,1,0,*]$, or $[1,1,1,0,*]$.

The set of midpoints $M_{\mathcal D(F)}$ contains all 8 classes of the form $\langle *,*,0,0,* \rangle$ in $\Z^5_{1/2}$ (here each star is $0$ or $\frac 12$ independently of others). The midpoints of segments connecting $A_i$ with vertices of the prism $XYZX'Y'Z'$ represent 6 of 8 classes of the form $\langle *,*,\frac 12,0,* \rangle$ except $\langle \frac 12,\frac 12,\frac 12,0,0 \rangle$ and $\langle \frac 12,\frac 12,\frac 12,0, \frac 12 \rangle$. These last two classes are present in the midpoints connecting $A_i'$ with the vertices of $XYZX'Y'Z'$ for every possible choice of $A_i'$.

Thus, all points of the four-dimensional linear space $x_4=0$ in $\Z^5_{1/2}$ are in $M_{\mathcal D(e_i)}$ and $e_i$ is a free direction for $P$ according to Lemma \ref{lem:free_space}.

{\bf Case \ref{lem:triangle}.2:} Each intersection $\mathcal D(e_i)\cap \pi_F'$ and $\mathcal D(e_j)\cap \pi_F'$ contains exactly two points that differ in the coordinate $x_5$. Since one point of the first intersection is $A_i\in[0,0,1,0,0]$, then the second one is $A_i'\in[0,0,1,0,1]$. We can also assume that $\mathcal D(e_j)\cap \pi_F'$ contains exactly two points from parity classes $A_j$ and $A_j'$ and $A_j'\in A_j+[0,0,0,0,1]$ as we can use {\bf Case \ref{lem:triangle}.1} for the edge $e_j$ in case this intersection contains more than two points, or two points with another difference in $\Z^5_p$.

We first look at the dual cell $\mathcal D(e_i)$. The midpoints of $A_iX$ and $A_i'X'$ represent the same class $\langle 0,0,\frac 12,0,0\rangle \in\Z^5_{1/2}$. If these midpoints coincide then $\overrightarrow{A_iA_i'}=\pm\overrightarrow{XX'}$. Otherwise, we use Lemma \ref{lem:translated_pair}  and move the midpoint of $A_i'X'$ onto the midpoint of $A_iX$ and one of the points $A_i'$ or $X'$ must move onto $A_i$ as this is the only point in the plane $\{x_3=1\}\subset \Z^5_p$ of $\mathcal D(e_i)$ other than $A_i'$. If $A_i'$ is moved into $A_i$ then $\overrightarrow{A_iA_i'}=\overrightarrow{XX'}$. If $X'$ is moved into $A_i$ then the midpoints of $A_iA_i'$ and $XX'$ coincide in $\R^5$ and similar argumets applied to midpoints of $A_iY$ and $A_i'Y'$ lead to a contradiction as the midpoints of $XX'$ and $YY'$ are different. Thus $\overrightarrow{A_iA_i'}=\pm\overrightarrow{XX'}$. Without loss of generality we can assume that $\overrightarrow{A_iA_i'}=\overrightarrow{XX'}$ as we can swap points $A_i$ and $A_i'$ and change coordinates ($x_3$ specifically) if needed.

Recall that $A_j\in[*,*,1,0,*]$. Since $\mathcal D(e_j)$ contains $A_j'$ as well we can assume that the fifth coordinate of $A_j$ is $0$. Below we consider all 4 cases for remaining pair of coordinates of $A_j$. By similar arguments we used above, $\overrightarrow{A_jA_j'}=\pm\overrightarrow{XX'}$

{\bf Subcase \ref{lem:triangle}.2.00:} $A_j\in[0,0,1,0,0]$ and $A_j'\in[0,0,1,0,1]$. We note that $A_j\neq A_i$ as in that case the copy of $P$ centered at $A_i$ contains two edges $e_i$ and $e_j$ of $F$ and it must contain $F$ as well, but this is false. 

The midpoints of $A_iX$ and $A_jX$ represent the class $\langle 0,0,\frac 12,0,0\rangle\in\Z^5_{1/2}$ so we can use Lemma \ref{lem:translated_pair} for the cell $\mathcal D(e_i)$ and points $A_j$ and $X$. The translations of $A_j$ and $X$ are in $\mathcal D(e_i)$ and they stay within the plane $x_4=0$ of $\Z^5_p$, so one of translations coincides with $A_i$ or $A_i'$ to get $x_3=\frac12$ for the midpoint.

If any of the points is translated into $A_i'$ then the second one is translated into the point symmetric to $X'$ with respect to $X$ (because $A_iXX'A_i'$ is a parallelogram), but this point does not belong to $\mathcal D(e_i)$. If $A_j$ is translated into $A_i$ then $A_j=A_i$ which is impossible. 

The only possible case is when $X$ is translated into $A_i$ and $A_j$ is translated into $X$. Thus $X$ is the midpoint of $A_iA_j$. The same arguments for midpoints of $A_iY$ and $A_jY$ lead to conclusion that $Y$ is the midpoint of $A_iA_j$ which is impossible.

{\bf Subcase \ref{lem:triangle}.2.10:} $A_j\in[1,0,1,0,0]$ and $A_j'\in[1,0,1,0,1]$. The midpoints of $A_iY$ and $A_jX$ represent the same class $\langle \frac 12,0, \frac12,0,0\rangle\in \Z^5_{1/2}$. We use Lemma \ref{lem:translated_pair} for the cell $\mathcal D(e_i)$ and points $A_j$ and $X$. One of the points $A_j$ or $X$ is translated into $A_i$ or $A_i'$. If any of the translations coincides with $A_i'$ then the second point is translated in the point symmetric to $Y'$ with respect to $Y$ ($A_iYY'A_i'$ is a parallelogram), and this point is not in $\mathcal D(e_i)$. 

If $A_j$ or $X$ is translated into $A_i$ then $\overrightarrow{A_jX}=\pm \overrightarrow {A_iY}$. Similar arguments for the midpoints of $A_iX$ and $A_jY$ that represent the class $\langle 0,0, \frac12,0,0\rangle\in \Z^5_{1/2}$ give that $\overrightarrow{A_jY}=\pm \overrightarrow {A_iX}$. Both equalities can be realized simultaneously only if the midpoints of $A_iA_j$ and $XY$ coincide.

Similar arguments for the midpoints of $A_i'Y'$ and $A_j'X'$ and the midpoints of $A_i'X'$ and $A_j'Y'$ show that the midpoints of $A_i'A_j'$ and $X'Y'$ coincide. Therefore $\overrightarrow{A_jA_j'}=\overrightarrow{A_iA_i}=\overrightarrow{XX'}$ and the midpoints of $A_iA_j'$, $A_i'A_j$ and $XY'$ coincide.

The parallelogram $XYY'X'$ is a dual cell of the tiling $\mathcal T_P$ (it is a face of the prismatic dual 3-cell $XYZX'Y'Z'$). Let $G$ be the 3-dimensional face of $\mathcal T_P$ with $\mathcal D(G)=XYY'X'$. The face $G$ is centrally symmetric with respect to the midpoint of $XY'$. Also, the face $G$ contains $F$ and hence $G$ contains edges $e_i$ and $e_j$. 

Let $e_j'$ be the edge of $G$ symmetric to $e_j$. The dual cell $\mathcal D(e_j')$ is centrally symmetric to $\mathcal D(e_j)$ with respect to the midpoint of $XY'$ and hence $\mathcal D(e_j')$ contains $X$, $Y$, $X'$, $Y'$, $A_i$, and $A_i'$. Let $H$ be the minimal face of $G$ that contains $e_i$ and $e_j'$. The dual cell of $H$ is the intersection of the dual cells $\mathcal D(e_i)$ and $\mathcal D(e_j')$ according to Lemma \ref{lem:intersection} so $\{X,Y,X',Y',A_i,A_i'\}\subseteq \mathcal D(H)$. Thus $\mathcal D(H)$ contains $\mathcal D(G)$ as a proper subset. Since $H$ contains at least two different edges $e_j'$ and $e_i$,  $H$ is a two-dimensional face of $G$.

The edges $e_j$ and $e_j'$ are parallel. If edges $e_i$ and $e_j$ are not parallel, then the line containing $e_i$ intersects both lines containing $e_j$ and $e_j'$, so the two-dimensional planes of faces $F$ and $H$ coincide. This is impossible as $e_j$ and $e_j'$ are opposite edges of $G$ and hence cannot belong to one supporting plane of $G$. Thus, $e_i$ and $e_j$ are parallel. The union of sets of midpoints $M_{\mathcal D(e_i)}$ and $M_{\mathcal D(e_j)}$ contains all classes within $\Z^5_{1/2}$ satisfying $x_4=0$. The arguments similar to the proof of Lemma \ref{lem:free_space} show that every triangular dual 2-face has a facet parallel to $e_i$ (and $e_j$). Lemma \ref{lem:free} implies that edge $e_i$ is a free direction of $P$.

{\bf Subcase \ref{lem:triangle}.2.01:} $A_j\in[0,1,1,0,0]$ and $A_j'\in[0,1,1,0,1]$. This subcase becomes identical to {\bf Subcase \ref{lem:triangle}.2.10} if we swap $Y$ and $Y'$ with $Z$ and $Z'$.

{\bf Subcase \ref{lem:triangle}.2.11:} $A_j\in[1,1,1,0,0]$ and $A_j'\in[1,1,1,0,1]$. This subcase becomes identical to {\bf Subcase \ref{lem:triangle}.2.10} if we swap $X$ and $X'$ with $Z$ and $Z'$.

{\bf Case \ref{lem:triangle}.3:} One of the intersection $\mathcal D(e_i)\cap \pi_F'$ and $\mathcal D(e_j)\cap \pi_F'$ contains exactly one point; without loss of generality we can assume that this intersection is $\mathcal D(e_i)\cap \pi_F'$ and $A_i\in[0,0,1,0,0]$. Recall that $A_j\in[*,*,1,0,*]$; below we consider all 8 cases for unknown coordinates in the parity class of $A_j$.

In most cases below we translate a segment within $\mathcal D(e_j)$ with one endpoint $A_j$ and the other endpoint in $XYZX'Y'Z'$ into the cell $\mathcal D(e_i)$ using Lemma \ref{lem:translated_pair}. Since this segment is not parallel to $\pi_F$ but parallel to $x_4=0$, the translation must have $A_i$ as one of the endpoints.

{\bf Subcase \ref{lem:triangle}.3.000:} $A_j\in[0,0,1,0,0]$. The midpoints of $A_jX$ and $A_iX$ represent the class $\langle 0,0,\frac 12,0,0\rangle\in\Z^5_{1/2}$; using Lemma \ref{lem:translated_pair} for the cell $\mathcal D(e_i)$ and points $A_j$ and $X$ we get that translated copy of $A_jX$ is within $\mathcal D(e_i)$. One of translated points coincides with $A_i$. This cannot be $A_j$ as in that case $A_i=A_j$ and the copy of $P$ centered at $A_i$ would contain two edges of $F$ but not $F$ itself. Thus translation of $X$ is $A_i$ and translation of $A_j$ is $X$. So $X$ is the midpoint of $A_iA_j$. Similar arguments for the midpoints of $A_iY$ and $A_jY$ give that $Y$ is the midpoint of $A_iA_j$ which is a contradiction.

{\bf Subcase \ref{lem:triangle}.3.100:} $A_j\in[1,0,1,0,0]$. The midpoints of $A_jX$ and $A_iY$ represent the class $\langle \frac 12,0,\frac 12,0,0\rangle\in \Z^5_{1/2}$. Using Lemma \ref{lem:translated_pair} we get that the segment $A_jX$ is translated into the segment $A_iY$ so $\overrightarrow{A_jX}=\pm\overrightarrow{A_iY}$. Similar arguments for midpoints of $A_jY$ and $A_iX$ give that $\overrightarrow{A_jY}=\pm\overrightarrow{A_iX}$. Both equalities can be realized simultaneously only if the midpoints of $XY$ and $A_iA_j$ coincide.

Similar arguments for the midpoints of $A_jX'$ and $A_iY'$ and for the midpoints of $A_jY'$ and $A_iX'$ give that the midpoints of $X'Y'$ and $A_iA_j$ coincide which is impossible.

{\bf Subcase \ref{lem:triangle}.3.010:} $A_j\in[0,1,1,0,0]$. This subcase becomes identical to {\bf Subcase \ref{lem:triangle}.3.100} if we swap $Y$ to $Z$ and $Y'$ to $Z'$.

{\bf Subcase \ref{lem:triangle}.3.110:} $A_j\in[1,1,1,0,0]$. This subcase becomes identical to {\bf Subcase \ref{lem:triangle}.3.100} if we swap $X$ to $Z$ and $X'$ to $Z'$.

{\bf Subcase \ref{lem:triangle}.3.001:} $A_j\in[0,0,1,0,1]$. The midpoints of $A_jX$ and $A_iX'$ represent the same class $\langle 0,0,\frac12,0,\frac 12\rangle \in \Z^5_{1/2}$, therefore $\overrightarrow{A_jX}=\pm\overrightarrow{A_iX'}$. Also the midpoints of $A_jX'$ and $A_iX$ represent the same class $\langle 0,0,\frac12,0,0\rangle \in \Z^5_{1/2}$, therefore $\overrightarrow{A_jX'}=\pm\overrightarrow{A_iX}$. This is possible only if the midpoints of $XX'$ and $A_iA_j$ coincide.

Similar arguments for the midpoints of $A_jY$ and $A_iY'$ and for the midpoints of $A_jY'$ and $A_iY$ give that the midpoints of $YY'$ and $A_iA_j$ coincide which is impossible.

{\bf Subcase \ref{lem:triangle}.3.101:} $A_j\in[1,0,1,0,1]$. The midpoints of $A_jX$ and $A_iY'$ represent the same class $\langle \frac 12,0,\frac 12,0,\frac 12\rangle \in \Z^5_{1/2}$ and therefore $\overrightarrow{A_jX}=\pm\overrightarrow{A_iY'}$. Similarly, the midpoints of $A_jY'$ and $A_iX$ represent the same class $\langle 0,0,\frac 12,0,0\rangle \in \Z^5_{1/2}$ and therefore $\overrightarrow{A_jX}=\pm\overrightarrow{A_iY'}$. This is possible only if the midpoints of $A_iA_j$ and $XY'$ coincide.

After that we use the same idea as in {\bf Subcase \ref{lem:triangle}.2.10}. Let $G$ be the three-dimensional face of $\mathcal T_P$ with dual cell $XYY'X'$. Let $e_j'$ be the edge of $G$ symmetric to $e_j$. The dual cell $\mathcal D(e_j')$ contains points $X, Y, Y', X'$ and $A_i$. Let $H$ be the minimal face of $G$ that contains $e_i$ and $e_j'$. By similar arguments, $H$ is a two-dimensional face of $G$ different from $F$. Again, similar arguments show that edges $e_i$ and $e_j$ are parallel.

As in {\bf Subcase \ref{lem:triangle}.2.10}, the union of the sets of midpoints $M_{\mathcal D(e_i)}\cup M_{\mathcal D(e_j)}$ contains all classes of $\Z^5_{1/2}$ satisfying $x_4=0$. This implies that $e_i$ is a free direction of $P$.

{\bf Subcase \ref{lem:triangle}.3.011:} $A_j\in[0,1,1,0,1]$. This subcase becomes identical to {\bf Subcase \ref{lem:triangle}.3.101} if we swap $Y$ to $Z$ and $Y'$ to $Z'$.

{\bf Subcase \ref{lem:triangle}.3.111:} $A_j\in[1,1,1,0,1]$. This subcase becomes identical to {\bf Subcase \ref{lem:triangle}.3.101} if we swap $X$ to $Z$ and $X'$ to $Z'$.

As we see, if $F$ is not a triangle, then in all possible cases $P$ has one edge of $F$ as a free direction. This concludes our proof.
\end{proof}

For the remaining part of the paper, $F$ is a triangle $xyz$ as $P$ has a free direction and satisfies the Voronoi conjecture otherwise. The following two corollaries follow the ideas of the proof of Lemma \ref{lem:triangle} and limit the options for each dual cell and what could be ``additional'' vertices within each of these three dual cells. After that, in the remaining sections we consider all possible cases for dual cells of edges $xy$, $yz$, and $zx$ of $F$. 

\begin{lemma}\label{lem:edge_cell}
Let $e$ be an edge of $F=xyz$ with dual cell $\mathcal D(F)=XYZX'Y'Z'$. The parallelohedron $P$ has a free direction or the dual 4-cell $\mathcal D(e)$ is equivalent to a prism over tetrahedron that has $XYZ$ as its face or to a pyramid over $XYZX'Y'Z'$.
\end{lemma}
\begin{proof}
As in the proof of Lemma \ref{lem:triangle}, let $\pi_F$ be the three-dimensional affine subspace of $\Z^5_p$ spanned by $XYZX'Y'Z'$. Each of dual cells $\mathcal D(xy)$, $\mathcal D(xz)$, and $\mathcal D(yz)$ have a point outside of $\pi_F$. If such points for two dual cells fall in one translated copy of $\pi_F$, then we use the same arguments as in the proof of Lemma \ref{lem:triangle} to show that $P$ has a free direction. Otherwise, the dual cells have additional points in different translations of $\pi_F$.

Without loss of generality suppose $e=xy$. Recall that $\mathcal D(e)$ has exactly six points in $\pi_F$. If $\mathcal D(e)$ has more than two points outside of $\pi_F$ or it has two points in $\pi_F$ with parity classes differing by a vector other than $\pm\overrightarrow{XX'}$, then we can use the arguments from {\bf Case \ref{lem:triangle}.1} of Lemma \ref{lem:triangle} to show that $e$ is a free direction for $P$. Thus, we have only two options $\mathcal D(e)=AXYZA'X'Y'Z'$ ($A$ and $A'$ are outside of $\pi_F$ and $\overrightarrow{AA'}=\overrightarrow{XX'}$ in $\Z^5_p$) or $\mathcal D(e)=AXYZX'Y'Z'$ ($A$ is outside $\pi_F$). 

{\bf Case \ref{lem:edge_cell}.1}: $\mathcal D(e)=AXYZA'X'Y'Z'$. First, we use the arguments from {\bf Case \ref{lem:triangle}.2} of Lemma \ref{lem:triangle} to show that $\overrightarrow{AA'}=\overrightarrow{XX'}$ in $\Z^5$ after possible swap of $A$ and $A'$. So geometrically, $\mathcal D(e)$ is a prism over tetrahedron $AXYZ$ and we need to recover the dual (sub)cells within $\mathcal D(e)$ to complete the proof for that case.

The parallelogram $XYY'X'$ is a dual 2-cell of a three-dimensional face of $\mathcal T_P$. It belongs to two dual 3-cells of two-dimensional faces of $\mathcal T_P$ within $AXYZA'X'Y'Z'$; one of these cells is $XYZX'Y'Z'$. Let $D_1$ denote the second cell that contains the cell $XYY'X'$. The cell $D_1$ contains either $A$ or $A'$ as $D_1$ cannot be a subset of $XYZX'Y'Z'$. The midpoints of $AX'$ and $A'X$ coincide, so in both cases the second point belongs to $D_1$ due to Lemma \ref{lem:translated_pair}. By similar reasons, if $D_1$ contains $Z$ or $Z'$ then it contains the other point and $D_1=\mathcal D(e)$ which is impossible, hence $D_1=AXYA'X'Y'$ is a cell equivalent to a prism over triangle. By similar reasons, the prisms $AXZA'X'Z'$ and $AYZA'Y'Z'$ are also subcells of $\mathcal D(e)$. 

The triangular cell $XYZ$ belongs to two dual 3-cells in $\mathcal D(e)$ as well; one of these 3-cells is $XYZX'Y'Z'$. Let $D_2$ be the second cell. If the cell $D_2$ contains $A'$, then it contains $A$ and $X'$ as well (midpoints of $AX'$ and $A'X$ coincide), but $D_2\cap XYZX'Y'Z'=XYZ$. Hence $D_2$ contains $A$ only and $D_2=AXYZ$, a dual cell equivalent to a tetrahedron. By similar reasons there is a tetrahedral dual 3-cell $A'X'Y'Z'$ within $\mathcal D(e)$. 

Summarizing, we found the following dual 3-cells within $\mathcal D(e)$: $AXYZ$, $A'X'Y'Z'$, $XYZX'Y'Z'$, $AXYA'X'Y'$, $AXZA'X'Z'$, and $AYZA'Y'Z'$. In this list every dual 2-cell belongs to exactly 2 dual 3-cells, hence it is a complete list of dual 3-cells within $\mathcal D(e)$. Thus, $\mathcal D(e)$ is equivalent to a prism over tetrahedron $AXYZ$.
{\sloppy

}

{\bf Case \ref{lem:edge_cell}.2}: $\mathcal D(e)=AXYZX'Y'Z'$. Similarly to the previous case we conclude that $AXYY'X'$ is the subcell (equivalent to a pyramid over parallelogram) of $\mathcal D(e)$ adjacent to $XYZX'Y'Z'$ by the parallelogram cell $XYY'X'$. By the similar reasons the pyramidal cells $AXZZ'X'$ and $AYZZ'Y'$ are subcells of $\mathcal D(e)$.

Also by similar reasons, the cells $AXYZ$ and $AX'Y'Z'$ are the only options for the second subcells of $\mathcal D(e)$ containing $XYZ$ and $X'Y'Z'$, respectively. The complete list of dual 3-cells within $\mathcal D(e)$ now looks as $AXYZ$, $AX'Y'Z'$, $XYZX'Y'Z'$, $AXYY'X'$, $AXZZ'X'$, and $AYZZ'Y'$. Thus, $\mathcal D(e)$ is equivalent to a pyramid over the prism $XYZX'Y'Z'$.
\end{proof}

Before formulating the next lemma we fix coordinate notations for parity classes of some points we have so far. We use these notations in the next lemma and in the next three sections. Recall that $F=xyz$ is a two-dimensional face of $P$ with dual cell $\mathcal D(xyz)=XYZX'Y'Z'$. Let $A,B,C\notin XYZX'Y'Z'$ be three points such that $A\in \mathcal D(xy)$, $B\in \mathcal D(xz)$, and $C\in \mathcal D(yz)$.

Without loss of generality we can assume that the points belong to the following parity classes in $\Z^5_p$
\begin{center}
\begin{tabular}{ll}
$X\in[0,0,0,0,0]$,\qquad \qquad & $X'\in[0,0,0,0,1]$,\\
$Y\in[1,0,0,0,0]$,\qquad \qquad & $Y'\in[1,0,0,0,1]$,\\
$Z\in[0,1,0,0,0]$,\qquad \qquad & $Z'\in[0,1,0,0,1]$.\\
\end{tabular}
\end{center}
In that case the affine span of $\mathcal D(F)$ in $\Z^5_p$ is given by $x_3=x_4=0$. 	So the points $A$, $B$, and $C$ belong to affine planes $x_3=1,x_4=0$; $x_3=0, x_4=1$; and $x_3=x_4=1$ in $\Z^5_p$ (maybe not respectively). 

Without loss of generality we can assume that $A\in [0,0,1,0,0]$ and the dual cell $\mathcal D(xy)$ may contain $A'\in [0,0,1,0,1]$ such that $\overrightarrow{AA'}=\overrightarrow{XX'}$. This follows from the possible structures of the dual cell $\mathcal D(xy)$ described in Lemma \ref{lem:edge_cell} after possible swap of $A$ and $A'$ and change of coordinates. Similarly, we can assume that $B\in [0,0,0,1,0]$ and the dual cell $\mathcal D(xz)$ may contain $B'\in [0,0,0,1,1]$ such that $\overrightarrow{BB'}=\overrightarrow{XX'}$. Finally, $C\in [*,*,1,1,*]$ and the dual cell $\mathcal D(yz)$ may contain $C'\in [*,*,1,1,*]$ such that $\overrightarrow{CC'}=\pm\overrightarrow{XX'}$.

The next lemma eliminates 6 options for $C$ leaving only 2. Particularly, in the previous notations, $C\in [1,1,1,1,0]$ or $C\in [1,1,1,1,1]$.

\begin{lemma}\label{lem:ABC}
Suppose $P$ does not have a free direction. Let $F=xyz$ be a face of $P$ with prismatic dual cell $XYZX'Y'Z'$.

Let $A \in \mathcal D(xy)$, $B \in \mathcal D(xz)$, and $C \in \mathcal D(yz)$ be three points in the corresponding dual cells that are not in $XYZX'Y'Z'$. Then $A+B+C$ represents the parity class of $X+Y+Z$ or $X'+Y'+Z'$.
\end{lemma}
\begin{proof}
Recall that $C\in [*,*,1,1,*]$. Below we show that 6 of 8 cases are impossible.

{\bf Case \ref{lem:ABC}.000}: $C\in[0,0,1,1,0]$. The points $A$ and $C$ belong to the dual cell $\mathcal D(y)$ and $BX$ is an edge (facet vector) of the dual cell $\mathcal D(xz)$ regardless of the type of dual cell of $xz$ from Lemma \ref{lem:edge_cell}. The midpoints of $AC$ and $BX$ represent the class $\langle 0,0,0,\frac 12,0\rangle \in\Z^5_{1/2}$ and using Lemma \ref{lem:translated_pair} for the cell $BX$ and points $A$ and $C$ we get $\overrightarrow{AC}=\pm\overrightarrow{BX}$. Similarly, $AX$ is an edge of $\mathcal D(xy)$ and $B,C\in \mathcal D(z)$ and midpoints of $AX$ and $BC$ represent $\langle 0,0,\frac 12,0,0 \rangle\in \Z^5_{1/2}$, and therefore $\overrightarrow{BC}=\pm\overrightarrow{AX}$. This is only possible if the midpoints of $AB$ and $CX$ coincide.

The cell $\mathcal D(yz)$ contains points $C$ and $X$ so Lemma \ref{lem:translated_pair} for this cell and points $A,B\in \mathcal D(x)$ implies that $\mathcal D(yz)$ contains $A$ and $B$ which is false.

{\bf Case \ref{lem:ABC}.100}: $C\in[1,0,1,1,0]$. This case becomes identical to {\bf Case \ref{lem:ABC}.000} if we swap $X$ with $Y$.

{\bf Case \ref{lem:ABC}.010}: $C\in[0,1,1,1,0]$. This case becomes identical to {\bf Case \ref{lem:ABC}.000} if we swap $X$ with $Z$.

{\bf Case \ref{lem:ABC}.001}: $C\in[0,0,1,1,1]$. Again, the points $A$ and $C$ belong to the dual cell $\mathcal D(y)$ and the midpoints of $AC$ and $BX'$ represent the same class $\langle 0,0,0,\frac 12,\frac 12\rangle\in\Z^5_{1/2}$. We use Lemma \ref{lem:translated_pair} for the dual cell $\mathcal D(xz)$ (which contains $B$ and $X'$) and points $A$ and $C$. After the translation the points $A$ and $C$ will be within $\mathcal D(xz)$. This is possible only if $\overrightarrow{AC}=\pm\overrightarrow{BX'}$ or $\mathcal D(xz)$ contains $B'$ and $\overrightarrow{AC}=\pm\overrightarrow{B'X}$.

If $A'\notin \mathcal D(xy)$ and $B'\notin \mathcal D(xz)$, then $\overrightarrow{AC}=\pm \overrightarrow{BX'}$ and by similar reasons $\overrightarrow{BC}=\pm \overrightarrow{AX'}$ because these midpoints represent the class $\langle 0,0,\frac 12,0,\frac 12\rangle\in\Z^5_{1/2}$, so the midpoints of $AB$ and $CX'$ coincide. The dual cell $\mathcal D(yz)$ contains $C$ and $X'$, so by Lemma \ref{lem:translated_pair} it must contain $A$ and $B$ (both in $\mathcal{D}(x)$) as well which is false.

If $A'\in \mathcal D(xy)$ and $B'\notin \mathcal D(xz)$, then $\overrightarrow{AC}=\pm \overrightarrow{BX'}$. By similar reasons $\overrightarrow{A'C}=\pm \overrightarrow{BX}$ because the midpoints of $A'C$ and $BX$ represent the class $\langle 0,0,0,\frac 12,0\rangle\in\Z^5_{1/2}$. This is possible only if the midpoints of $BC$ and $AX'$ coincide. However the midpoint of $AX'$ in that case belongs to the contact face with dual cell $AXX'A'$, so it cannot belong to either copy of $P$ centered at $B$ or $C$. Similar reasons work if $A'\notin \mathcal D(xy)$ and $B'\in \mathcal D(xz)$.

The last case is if $A'\in \mathcal D(xy)$ and $B'\in \mathcal D(xz)$. Then the midpoints of $A'C$ and $BX$ represent the class $\langle 0,0,0,\frac 12,0\rangle\in\Z^5_{1/2}$, so $\overrightarrow{A'C}=\pm \overrightarrow{BX}$ because $BX$ is an edge of the dual cell $\mathcal D(xz)$. Similarly $\overrightarrow{B'C}=\pm \overrightarrow{AX}$. This is possible only if the midpoints of $CX$ and $A'B$ coincide. The dual cell $\mathcal D(yz)$ contains $C$ and $X$, so by Lemma \ref{lem:translated_pair} it must contain $A'$ and $B$ as well which is false. 

{\bf Case \ref{lem:ABC}.101}: $C\in[1,0,1,1,1]$. This case becomes identical to {\bf Case \ref{lem:ABC}.001} if we swap $X$ to $Y$ and $X'$ to $Y'$.

{\bf Case \ref{lem:ABC}.011}: $C\in[0,1,1,1,1]$. This case becomes identical to {\bf Case \ref{lem:ABC}.001} if we swap $X$ to $Z$ and $X'$ to $Z'$.

The remaining two cases for $C$ are $C \in [1,1,1,1,0]$ and $C \in [1,1,1,1,1]$. For the first option $A+B+C$ and $X+Y+Z$ represent the parity class $[1,1,0,0,0]$ and for the second option $A+B+C$ and $X'+Y'+Z'$ represent the parity class $[1,1,0,0,1]$.
\end{proof}

In the next four sections we consider all possible cases for dual cells of edges of $F$ as outlined in the proof of Theorem \ref{thm:main}. Recall that $pr(F)$ denotes the number of dual 4-cells equivalent to prisms over tetrahedra among the dual cells of edges of $F$. We consider the case $pr(F)=3,2,1,0$ in the next sections, respectively.

\section{Prism-Prism-Prism case}\label{sec:pr-pr-pr}

In this case we assume that the dual cells of the edges $xy$, $xz$, and $yz$ of $F$ are prisms over tetrahedra; that is, we consider the case $pr(F)=3$. The results of Section \ref{sec:prism} imply that it is sufficient consider only the case $\mathcal D(xyz)=XYZX'Y'Z'$, $\mathcal D(xy)=AXYZA'X'Y'Z'$, $\mathcal D(xz)=BXYZB'X'Y'Z'$, and $\mathcal D(yz)=CXYZC'X'Y'Z'$ where 

$$\overrightarrow{XX'}=\overrightarrow{YY'}=\overrightarrow{ZZ'}=\overrightarrow{AA'}=\overrightarrow{BB'}=\pm\overrightarrow{CC'}$$
and the points represent the following parity classes
\begin{center}
\begin{tabular}{ll}
$X\in[0,0,0,0,0]$,\qquad \qquad & $X'\in[0,0,0,0,1]$,\\
$Y\in[1,0,0,0,0]$,\qquad \qquad & $Y'\in[1,0,0,0,1]$,\\
$Z\in[0,1,0,0,0]$,\qquad \qquad & $Z'\in[0,1,0,0,1]$,\\
$A\in[0,0,1,0,0]$,\qquad \qquad & $A'\in[0,0,1,0,1]$,\\
$B\in[0,0,0,1,0]$,\qquad \qquad & $B'\in[0,0,0,1,1]$,\\
\end{tabular}
\end{center}
while $C$ is in $[1,1,1,1,1]$ or in $[1,1,1,1,0]$. Then $C'$ is in $[1,1,1,1,0]$ or in $[1,1,1,1,1]$, respectively.

We use the following ``red Venkov graph'' criterion for $P$ to be decomposable in a direct sum of two parallelohedra of smaller dimensions. This criterion was proved by Ordine in \cite{Ord}; we also refer to \cite{Mag} and \cite{MO} for details.

\begin{definition}
Let $P$ be a $d$-dimensional parallelohedron, $d\geq 2$. We construct the following graph $G_P$ based on the combinatorial structure of the tiling $\mathcal T_P$. 

The vertices of $G_P$ correspond to pairs of opposite facets of $P$ or, equivalently, to equivalence classes of facet vectors of $\mathcal T_P$. Two facet vectors are treated as equivalent if they are parallel. Two vertices $\alpha$ and $\beta$ of $G_P$ are connected with an edge there is a triangular dual 2-face of $\mathcal T_P$ with two edges equivalent to $\alpha$ and $\beta$.

The graph $G_P$ is called the {\it red Venkov graph} of $P$.
\end{definition}

We highlight the following connection between the red Venkov graph, the gain function introduced in Definition \ref{def:gain}, and canonical scaling from Definition \ref{def:scaling}. Two facet vectors $\alpha$ and $\beta$ of $\mathcal T_P$ are connected with an edge in the red Venkov graph if and only if the gain function $\gamma(\alpha,\beta)$ is defined for the two facet vectors as both properties require existence of a triangular dual cell of $\mathcal T_P$ that contains $\alpha$ and $\beta$. We recall that the gain function tracks how a hypothetical canonical scaling would change. Therefore, the red Venkov graph is connected if and only if a value of hypothetical canonical scaling on any one facet of $P$ defines the values on all other facets. We use connection between the gain function and canonical scaling in Section \ref{sec:py-py-py} and in this section we need the following property of the red Venkov graph.


\begin{theorem}[A.~Ordine, {\cite{Ord}}]\label{lem:red}
A parallelohedron $P$ is a direct sum of two parallelohedra of smaller dimension if and only if the graph $G_P$ is disconnected.
\end{theorem}

\begin{lemma}\label{lem:pr-pr-pr}
If dual cells $\mathcal D(xy)$, $\mathcal D(xz)$, and $\mathcal D(yz)$ are equivalent to prisms over tetrahedra, then $P$ satisfies the Voronoi conjecture.
\end{lemma}
\begin{proof}
We claim that the vertex of $G_P$ corresponding to the facet vector $XX'$ is an isolated vertex. 

Suppose $XX'$ corresponds to a non-isolated vertex of $G_P$. Then $XX'$ belongs to a triangular dual 2-cell $TXX'$ of $\mathcal T_P$ for some $T\in \Z^5$. The midpoints of facet vectors $TX$ and $TX'$ represent non-zero classes of $\Z^5_{1/2}$ and differ by $\langle 0,0,0,0,\frac12\rangle$. Below we show that for every choice of $a,b,c,d\in\{0,\frac 12\}$ except $a=b=c=d=0$ one of two classes of the form $\langle a,b,c,d,*\rangle$ does not represent a facet vector which gives a contradiction.

The first family of midpoints that are not midpoints of facet vectors comes from parallelogram dual 2-cells within $\mathcal D(xy)=AXYZA'X'Y'Z'$ and $\mathcal D(xz)=BXYZB'X'Y'Z'$.
\begin{center}
\begin{tabular}{ll}
$\langle \frac12,0,0,0,0\rangle$ & is the center of parallelogram cell $XYY'X'$,\\
$\langle 0,\frac12,0,0,0\rangle$ & is the center of parallelogram cell $XZZ'X'$,\\
$\langle \frac12,\frac12,0,0,0\rangle$ & is the center of parallelogram cell $YZZ'Y'$,\\
$\langle 0,0,\frac12,0,0\rangle$ & is the center of parallelogram cell $AXX'A$,\\
$\langle \frac12,0,\frac12,0,0\rangle$ & is the center of parallelogram cell $AYY'A'$,\\
$\langle 0,\frac12,\frac12,0,0\rangle$ & is the center of parallelogram cell $AZZ'A'$,\\
$\langle 0,0,0,\frac12,0\rangle$ & is the center of parallelogram cell $BXX'B$,\\
$\langle \frac12,0,0,\frac12,0\rangle$ & is the center of parallelogram cell $BYY'B'$,\\
$\langle 0,\frac12,0,\frac12,0\rangle$ & is the center of parallelogram cell $BZZ'B'$.\\
\end{tabular}
\end{center}

Similarly, four points $C$, $X$, $X'$, and $C'$ form a parallelogram dual 2-cell within $\mathcal D(yz)=CXYZC'X'Y'Z'$ with the center being the midpoint of $CX$ or the midpoint of $CX'$. In both cases the center does not correspond to a facet vector and has the form $\langle \frac12,\frac12,\frac12,\frac12,*\rangle$. This and two similar parallelograms give three more non-facet midpoints.
\begin{center}
\begin{tabular}{ll}
$\langle \frac12,\frac12,\frac12,\frac12,*\rangle$ & is the center of parallelogram cell with vertices $C$, $X$, $X'$, and $C'$,\\
$\langle 0,\frac12,\frac12,\frac12,*\rangle$ & is the center of parallelogram cell with vertices $C$, $Y$, $Y'$, and $C'$,\\
$\langle \frac12,0,\frac12,\frac12,*\rangle$ & is the center of parallelogram cell with vertices $C$, $Z$, $Z'$, and $C'$.\\
\end{tabular}
\end{center}
The last block of three midpoints comes from points within one dual cell that geometrically form a parallelogram but not necessarily form a dual 2-cell equivalent to parallelogram. The points $A$, $B$, $B'$, and $A'$ all belong to the dual cell $\mathcal D(x)$ and form a parallelogram with center in the midpoint of $AB'$. If this midpoint corresponds to a facet vector, then both $AB'$ and $A'B$ must be facet vectors due to Lemma \ref{lem:translated_cell} which is impossible.
\begin{center}
\begin{tabular}{ll}
$\langle 0,0,\frac12,\frac12,0\rangle$ & is the center of parallelogram with vertices $A,B,B',A'\in \mathcal D(x)$,\\
$\langle \frac12,\frac12,0,\frac12,*\rangle$ & is the center of parallelogram with vertices $A,C,C',A'\in \mathcal D(y)$,\\
$\langle \frac12,\frac12,\frac12,0,*\rangle$ & is the center of parallelogram with vertices $B,C,C',B'\in \mathcal D(z)$.\\
\end{tabular}
\end{center}
Thus, we have treated all 15 cases for $a,b,c,d$.

This implies that $G_P$ is disconnected and hence by Theorem \ref{lem:red}, $P=P_1\oplus P_2$ for some parallelohedra $P_1$ and $P_2$ of dimensions at most 4. The Voronoi conjecture is true for $P_1$ and for $P_2$ and therefore the Voronoi conjecture holds for $P$.
\end{proof}

\begin{remark}
We note that since $P_1$ and $P_2$ are parallelohedra of dimensions at most $4$, both $P_1$ and $P_2$ have free directions. Consequently, in this case $P$ has a free direction as well.
\end{remark}

\section{Prism-Prism-Pyramid case}\label{sec:pr-pr-py}

In this case we assume that the dual cells of the edges $xy$ and $xz$ of $F$ are prisms over tetrahedra and the dual cell of edge $yz$ is a pyramid over triangular prism; this is the case $pr(F)=2$. The results of Section \ref{sec:prism} imply that it is sufficient to consider only the case $\mathcal D(xyz)=XYZX'Y'Z'$, $\mathcal D(xy)=AXYZA'X'Y'Z'$, $\mathcal D(xz)=BXYZB'X'Y'Z'$, and $\mathcal D(yz)=CXYZX'Y'Z'$ where 
$$\overrightarrow{XX'}=\overrightarrow{YY'}=\overrightarrow{ZZ'}=\overrightarrow{AA'}=\overrightarrow{BB'}$$
and the points represent the following parity classes
\begin{center}
\begin{tabular}{ll}
$X\in[0,0,0,0,0]$,\qquad \qquad & $X'\in[0,0,0,0,1]$,\\
$Y\in[1,0,0,0,0]$,\qquad \qquad & $Y'\in[1,0,0,0,1]$,\\
$Z\in[0,1,0,0,0]$,\qquad \qquad & $Z'\in[0,1,0,0,1]$,\\
$A\in[0,0,1,0,0]$,\qquad \qquad & $A'\in[0,0,1,0,1]$,\\
$B\in[0,0,0,1,0]$,\qquad \qquad & $B'\in[0,0,0,1,1]$,\\
\end{tabular}
\end{center}
while $C$ is in $[1,1,1,1,0]$ or in $[1,1,1,1,1]$. The goal of this section is to show that this configuration is impossible unless $P$ has a free direction.

\begin{lemma}\label{lem:pr-pr-py-10}
The dual cell $\mathcal D(x)$ contains exactly $10$ points so $\mathcal D(x)=ABXYZA'B'X'Y'Z'$.
\end{lemma}
\begin{proof}
Suppose $\mathcal D(x)$ contains an additional point $R$. The point $R$ cannot belong to a parity class of points $A$, $B$, $X$, $Y$, $Z$, $A'$, $B'$, $X'$, $Y'$, or $Z'$. Also we use Lemma \ref{lem:3+1} for 14 triangular dual 2-cells within the dual cell $\mathcal D(x)$ and each triangle forbids a parity class for $R$.
\begin{center}
\begin{tabular}{l|l}
Triangle & Forbidden parity class \\ \hline
$XYZ$ & $X+Y+Z\in [1,1,0,0,0]$\\
$AXY$ & $A+X+Y\in [1,0,1,0,0]$\\
$AXZ$ & $A+X+Z\in [0,1,1,0,0]$\\
$AYZ$ & $A+Y+Z\in [1,1,1,0,0]$\\
$BXY$ & $B+X+Y\in [1,0,0,1,0]$\\
$BXZ$ & $B+X+Z\in [0,1,0,1,0]$\\
$BYZ$ & $B+Y+Z\in [1,1,0,1,0]$\\
$X'Y'Z'$ & $X'+Y'+Z'\in [1,1,0,0,1]$\\
$A'X'Y'$ & $A'+X'+Y'\in [1,0,1,0,1]$\\
$A'X'Z'$ & $A'+X'+Z'\in [0,1,1,0,1]$\\
$A'Y'Z'$ & $A'+Y'+Z'\in [1,1,1,0,1]$\\
$B'X'Y'$ & $B'+X'+Y'\in [1,0,0,1,1]$\\
$B'X'Z'$ & $B'+X'+Z'\in [0,1,0,1,1]$\\
$B'Y'Z'$ & $B'+Y'+Z'\in [1,1,0,1,1]$\\
\end{tabular}
\end{center}

We have eliminated 24 options for the parity class of $R$ (all except 8 points in  the 3-dimensional plane $x_3=x_4=1$ in $\Z^5_p$). The remaining 8 options are studied below.

{\bf Case \ref{lem:pr-pr-py-10}.00110}: $R\in [0,0,1,1,0]$. The midpoints of $AX'$ and $B'R$ represent the parity class $\langle 0,0,\frac 12,0,\frac 12\rangle \in \Z^5_{1/2}$. Therefore, by Lemma \ref{lem:translated_cell}, the midpoint of $B'R$ corresponds to a contact dual 2-cell which is a translation of $AXX'A'$. Thus, the dual cell $\mathcal D(x)$ contains a point $R'\in [0,0,1,1,1]$. Moreover, since $B$ is the only point of its parity class in $\mathcal D(x)$, the translation of $AXX'A'$ is the parallelogram $RBB'R'$ centered at the midpoint of $B'R$ and $\overrightarrow{RR'}=\overrightarrow{BB'}$. 

The points $R$, $Y$, $Y'$, $R'$ form a parallelogram centered at the midpoint of $RY'\in \langle \frac 12,0,\frac 12,\frac 12,0\rangle$. This parallelogram is not necessarily a dual cell, but its center does not belong to a facet vector as in that case both diagonals of this parallelogram are facet vectors and facet vectors cannot intersect. However, the segments $CZ$ and $CZ'$ are facet vectors within the cell $\mathcal D(yz)$ and their midpoints are in classes $\langle \frac 12,0,\frac 12,\frac 12,0\rangle$ and $\langle \frac 12,0,\frac 12,\frac 12,\frac 12\rangle$ for both options for the parity class of $C$ which gives a contradiction.
{\sloppy

}

{\bf Case \ref{lem:pr-pr-py-10}.10110}: $R\in [1,0,1,1,0]$. This case becomes identical to {\bf Case \ref{lem:pr-pr-py-10}.00110} if we swap $X$ and $Y$ and $X'$ and $Y'$ for finding $R'$ and use a similar framework afterwards.

{\bf Case \ref{lem:pr-pr-py-10}.01110}: $R\in [0,1,1,1,0]$. This case becomes identical to {\bf Case \ref{lem:pr-pr-py-10}.00110} if we swap $X$ and $Z$ and $X'$ and $Z'$ for finding $R'$ and use a similar framework afterwards..

{\bf Case \ref{lem:pr-pr-py-10}.00111}: $R\in [0,0,1,1,1]$. The midpoints of $AX'$ and $BR$ represent the parity class $\langle 0,0,\frac 12,0,\frac 12\rangle \in \Z^5_{1/2}$. Thus, by the reasons similar to {\bf Case \ref{lem:pr-pr-py-10}.00110}, the cell $\mathcal D(x)$ contains a point $R'\in [0,0,1,1,0]$ which is impossible by {\bf Case \ref{lem:pr-pr-py-10}.00110}.

{\bf Case \ref{lem:pr-pr-py-10}.10111}: $R\in [1,0,1,1,1]$. This case becomes identical to {\bf Case \ref{lem:pr-pr-py-10}.00111} if we swap $X'$ to $Y'$ and use impossibility of {\bf Case \ref{lem:pr-pr-py-10}.10110}.

{\bf Case \ref{lem:pr-pr-py-10}.01111}: $R\in [0,1,1,1,1]$. This case becomes identical to {\bf Case \ref{lem:pr-pr-py-10}.00111} if we swap $X'$ to $Z'$ and use impossibility of {\bf Case \ref{lem:pr-pr-py-10}.01110}.

{\bf Case \ref{lem:pr-pr-py-10}.1111*}: $R\in [1,1,1,1,0]$ or $R\in [1,1,1,1,1]$. Note that these are the same options for parity classes we have for $C$.

If $C$ and $R$ belong to the same parity class, then midpoints of $CX$ (within $\mathcal D(yz)$) and $RX$ (within $\mathcal D(x)$) belong to the same class in $\Z^5_{1/2}$ and $CX$ is an edge of the cell $\mathcal D(yz)$. Thus by Lemma \ref{lem:translated_pair} $\overrightarrow{RX}=\pm\overrightarrow{CX}$. The points $R$ and $C$ are different as otherwise the copy of $P$ centered at $C$ contains all vertices of $F$ but does not contain $F$ itself, which is impossible. The only other option is when $X$ is the midpoint of $CR$. Similar arguments for midpoints of $CX'$ and $RX'$ show that $X'$ is the midpoint of $CR$ which is a contradiction.

If $C$ and $R$ are in different parity classes then their classes differ by $[0,0,0,0,1]$. Therefore the midpoints of $CX$ and $RX'$ represent the same class in $\Z^5_{1/2}$ and $CX$ is an edge of $\mathcal D(yz)$. Thus, $\overrightarrow{RX'}=\pm\overrightarrow{CX}$. Similar arguments for midpoints of $CX'$ and $RX$ give that $\overrightarrow{RX}=\pm\overrightarrow{CX'}$. This is only possible when midpoints of $XX'$ and $CR$ coincide. Same arguments for the midpoints of $CY$ and $RY'$ and for the midpoints of $CY'$ and $RY$ show that the midpoints of $CR$ and $YY'$ coincide which is a contradiction.
\end{proof}

\begin{lemma}\label{lem:pr-pr-py}
If a triangular face $xyz$ of $P$ with prismatic dual cell has exactly two edges $xy$ and $xz$ with dual cells equivalent to prisms over tetrahedra, then $P$ has a free direction.
\end{lemma}
\begin{proof}
Suppose $P$ does not have a free direction. $\mathcal D(x)=ABXYZA'B'X'Y'Z'$ according to Lemma \ref{lem:pr-pr-py-10}. It might be useful to use Figure \ref{pict:pr-py-py} to track dual cells of faces and edges we use in the arguments.

Parallelogram $XYY'X'$ is the dual cell of a face of the tiling $\mathcal T_P$. Let $G$ be this face, so $\mathcal D(G)=XYY'X'$ and $\dim G=3$. In particular, the triangle $F=xyz$ is a face of $G$. Let $H_{xy}$ be the face of $G$ adjacent to $F$ by $xy$. The dual cell of $H_{xy}$ contains the points $X$, $Y$, $Y'$ and $X'$ and is a subcell of $\mathcal D(xy)=AXYZA'X'Y'Z'$. The only such 3-cell other than $\mathcal D(F)$ is $AXYA'X'Y'$ so $\mathcal D(H_{xy})=AXYA'X'Y'$ which is equivalent to triangular prism. Thus $H_{xy}$ is a triangle or $P$ satisfies the Voronoi conjecture due to Lemma \ref{lem:triangle}.

\begin{center}
\begin{figure}[!ht]
\includegraphics[width=0.8\textwidth]{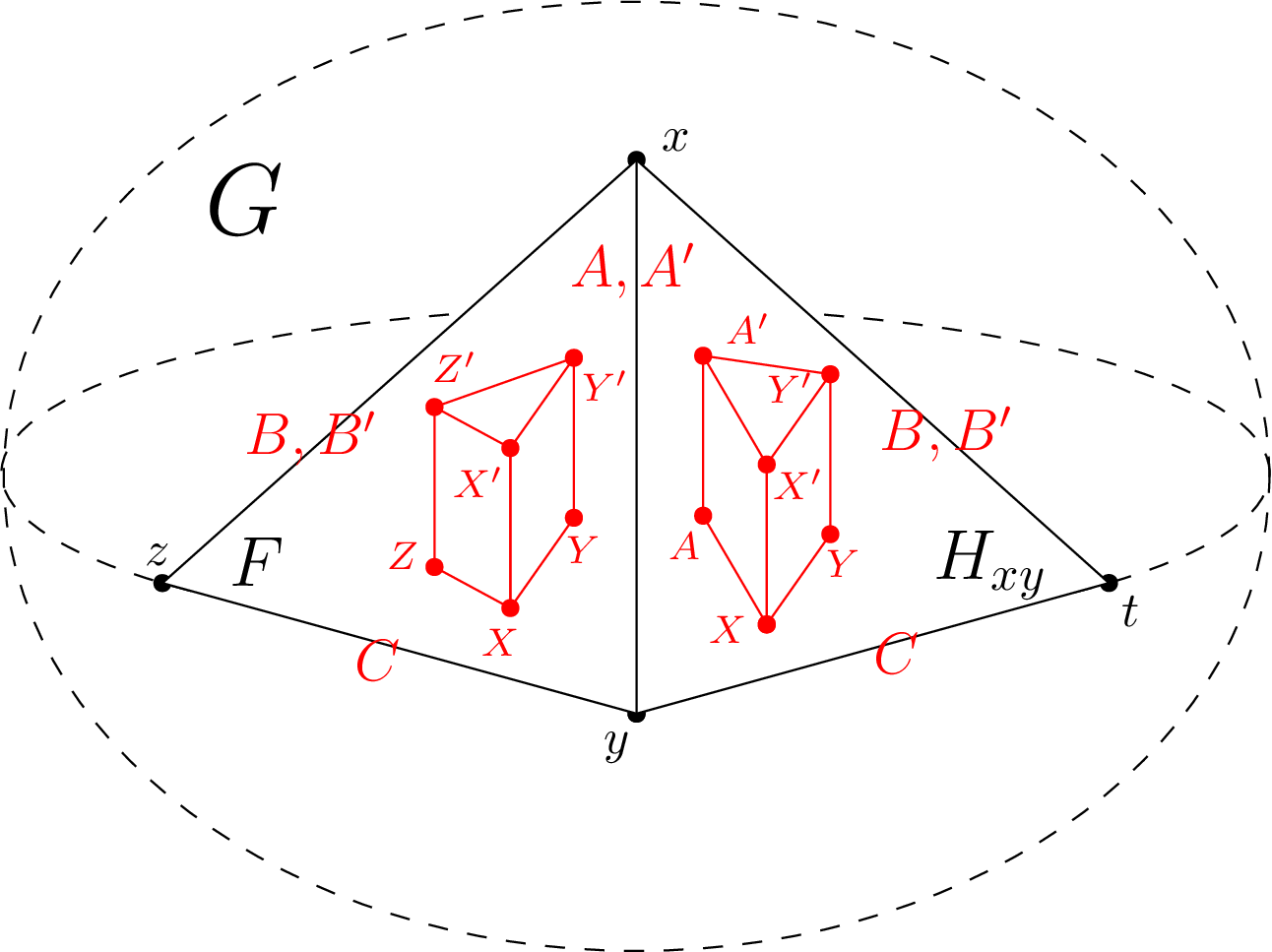}
\caption{An illustration for the proof of Lemma \ref{lem:pr-pr-py}. The face $G$ with $\mathcal D(G)=XYY'X$ and its triangular faces $xyz$ and $xyt$ with prismatic dual cells. We put dual cells of two-dimensional faces inside corresponding triangles and show only additional points corresponding to edges. Dual cells are shown in red.}
\label{pict:pr-pr-py}
\end{figure}
\end{center}

Let $H_{xy}=xyt$ for some point $t$. We look at the dual cell of the edge $xt$. This dual cell contains $AXYA'X'Y'$, the dual cell of $xyt$, and is contained in $ABXYZA'B'X'Y'Z'$, the dual cell of $x$ obtained in Lemma \ref{lem:pr-pr-py-10}. The dual cell $\mathcal D(xt)$ does not contain $Z$ (or $Z'$) as in that case the copy of $P$ centered at $Z$ (or $Z'$) would contain all three vertices of $xyt$ but not the triangle $xyt$ itself. Hence, $\mathcal D(xt)$ contains $B$ or $B'$. Since the intersection $\mathcal D(xz)=BXYZB'X'Y'Z'$ and $\mathcal D(xt)$ is a subcell of both, $\mathcal D(xt)$ contains both $B$ and $B'$ and $\mathcal D(xt)=ABXYA'B'X'Y'$.

The dual cells $\mathcal D(xz)=BXYZB'X'Y'Z'$ and $D(xt)=ABXYA'B'X'Y'$ intersect by the dual 3-cell $BXYB'X'Y'$ and therefore the edges $xz$ and $xt$ belong to a two-dimensional face $H_{xz}$ of $G$ with the dual cell $\mathcal D(H_{xz})=BXYB'X'Y'$. Since this dual cell is equivalent to a prism, the face $H_{xz}$ is a triangle and $H_{xz}=xzt$ unless $P$ has a free direction.

Next we identify the dual cell $\mathcal{D}(yt)$. This cell contains $AXYA'X'Y'$, the dual cell of $xyt$. According to Lemma \ref{lem:edge_cell}, the dual cell of $yt$ contains one additional vertex $R$ or two additional vertices $R$ and $R'$ that differ by $\overrightarrow{XX'}$. Lemma \ref{lem:ABC} for triangle $xyt$ with prismatic dual cell $AXYA'X'Y'$ implies that $Z+B+R=A+X+Y$ or $Z+B+R=A'+X'+Y'$ in $\Z^5_p$. This means that $R\in [1,1,1,1,*]\in \Z^5_p$.

The points $R$, $C$, $X$, and $X'$ are all in the dual cell $\mathcal D(y)$, and $CXX'$ is a triangular dual 2-cell within the cell $\mathcal D(yz)=CXYZX'Y'Z'$ equivalent to a pyramid over triangular prism. Hence by Lemma \ref{lem:3+1}, no point of the parity class $C+X+X'=C+[0,0,0,0,1]$ belong to $\mathcal D(y)$. Thus $R$ represent the same parity class as $C$ and the dual cell of $yt$ does not contain another point $R'$. Also $R=C$ as $\mathcal D(y)$ can contain only one point from the parity class of $C$. So $\mathcal D(yt)=CAXYA'X'Y'$.

Similarly to the edges $xz$ and $xt$, the dual cells $\mathcal D(yz)=CXYZX'Y'Z'$ and $\mathcal D(yt)=CAXYA'X'Y'$ intersect by the dual 3-cell $CXYX'Y'$ and therefore $yz$ and $yt$ belong to a two-dimensional face $H_{yz}$ of $G$ with $\mathcal D(H_{yz})=CXYX'Y'$. Two 2-dimensional faces $H_{yz}$ and $H_{xz}$ of $F$ have two vertices $z$ and $t$ in common, hence $zt$ is an edge of both and $H_{yz}=yzt$. This means that the face $G$ is a tetrahedron $xyzt$ as we identified four triangular faces $F=xyz$, $H_{xy}=xyt$, $H_{xz}=xzt$, and $H_{yz}=yzt$ of $G$.

However the dual cell $\mathcal D(G)$ is $XYY'X'$, so $G$ is a contact face and must be centrally symmetric. Hence $G$ cannot be a tetrahedron.
\end{proof}

\section{Prism-Pyramid-Pyramid case}\label{sec:pr-py-py}

In this case we assume that the dual cell of the edge $xy$ of $F$ is a prism over tetrahedron and the dual cells of edges $xz$ and $yz$ are pyramids over the  prism $XYZX'Y'Z'$; this is the case $pr(F)=1$. The results of Section \ref{sec:prism} imply that it is sufficient to consider only the case $\mathcal D(xyz)=XYZX'Y'Z'$, $\mathcal D(xy)=AXYZA'X'Y'Z'$, $\mathcal D(xz)=BXYZX'Y'Z'$, and $\mathcal D(yz)=CXYZX'Y'Z'$ where 
$$\overrightarrow{XX'}=\overrightarrow{YY'}=\overrightarrow{ZZ'}=\overrightarrow{AA'}$$
and the points represent the following parity classes
\begin{center}
\begin{tabular}{ll}
$X\in[0,0,0,0,0]$,\qquad \qquad & $X'\in[0,0,0,0,1]$,\\
$Y\in[1,0,0,0,0]$,\qquad \qquad & $Y'\in[1,0,0,0,1]$,\\
$Z\in[0,1,0,0,0]$,\qquad \qquad & $Z'\in[0,1,0,0,1]$,\\
$A\in[0,0,1,0,0]$,\qquad \qquad & $A'\in[0,0,1,0,1]$,\\
$B\in[0,0,0,1,0]$,\qquad \qquad & \\
\end{tabular}
\end{center}
while $C$ is in $[1,1,1,1,1]$ or in $[1,1,1,1,0]$. The goal of this section is to show that this configuration is impossible unless $P$ has a free direction; we use a framework similar to the one used in Section \ref{sec:pr-pr-py}.

\begin{lemma}\label{lem:pr-py-py-9}
The dual cell $\mathcal D(x)$ contains exactly $9$ points so $\mathcal D(x)=BAXYZA'X'Y'Z'$ or $P$ has a free direction.
\end{lemma}
\begin{proof}
Suppose $\mathcal D(x)$ contains an additional point $R$. The point $R$ cannot belong to a parity class of points $B$, $A$, $X$, $Y$, $Z$, $A'$, $X'$, $Y'$, or $Z'$. Also we use Lemma \ref{lem:3+1} for 12 triangular dual 2-cells within the dual cell $\mathcal D(x)$ and each triangle forbids a parity class for $R$.
\begin{center}
\begin{tabular}{l|l}
Triangle & Forbidden parity class \\ \hline
$XYZ$ & $X+Y+Z\in [1,1,0,0,0]$\\
$AXY$ & $A+X+Y\in [1,0,1,0,0]$\\
$AXZ$ & $A+X+Z\in [0,1,1,0,0]$\\
$AYZ$ & $A+Y+Z\in [1,1,1,0,0]$\\
$X'Y'Z'$ & $X'+Y'+Z'\in [1,1,0,0,1]$\\
$A'X'Y'$ & $A'+X'+Y'\in [1,0,1,0,1]$\\
$A'X'Z'$ & $A'+X'+Z'\in [0,1,1,0,1]$\\
$A'Y'Z'$ & $A'+Y'+Z'\in [1,1,1,0,1]$\\
$BXY$ & $B+X+Y\in [1,0,0,1,0]$\\
$BXZ$ & $B+X+Z\in [0,1,0,1,0]$\\
$BYZ$ & $B+Y+Z\in [1,1,0,1,0]$\\
$BXX'$ & $B+X+X'\in [0,0,0,1,1]$.
\end{tabular}
\end{center}

We have eliminated 21 possible cases for the parity class of $R$ and the remaining 11 cases are eliminated on the case-by-case basis.

{\bf Case \ref{lem:pr-py-py-9}.10011}: $R\in [1,0,0,1,1]$. The midpoints of segments $XY'$ and $RB$ represent the same class $\langle \frac12,0,0,0,\frac12\rangle\in\Z^5_{1/2}$. Using Lemma \ref{lem:translated_cell} for the cell $\mathcal D(x)$ and the contact face with dual cell $XYY'X'$ we get that $\mathcal D(x)$ contains a point $R'\in [1,0,0,1,0]$. This case was eliminated above using Lemma \ref{lem:3+1} for points $B$, $X$, and $Y$.

{\bf Case \ref{lem:pr-py-py-9}.01011}: $R\in [0,1,0,1,1]$. This case becomes identical to {\bf Case \ref{lem:pr-py-py-9}.10011} if we swap $Y$ to $Z$ and $Y'$ to $Z'$.

{\bf Case \ref{lem:pr-py-py-9}.11011}: $R\in [1,1,0,1,1]$. This case becomes identical to {\bf Case \ref{lem:pr-py-py-9}.10011} if we swap $X$ to $Z$ and $X'$ to $Z'$.

{\bf Case \ref{lem:pr-py-py-9}.00110}: $R\in [0,0,1,1,0]$. If $AB$ is a facet vector of $\mathcal T_P$ then we get a contradiction with Lemma \ref{lem:3+1} for the three points $A$, $B$, $X$ connected with facet vectors and $R\in A+B+X$ in $\Z^5_p$ as all four points belong to $\mathcal D(x)$. Similarly, the segment $A'B$ is not a facet vector as we get a contradiction with Lemma \ref{lem:3+1} for points $A'$, $B$, $X'$ and $R\in A'+B+X'$ otherwise.

The set of midpoints $M_{\mathcal D(yz)}$ of the dual cell $\mathcal D(yz)=CXYZX'Y'Z'$ contains 14 classes of points satisfying $x_3=x_4$ in $\Z^5_{1/2}$.  The remaining two points in this 4-dimensional space are $\langle 0,0,\frac 12,\frac 12,0\rangle$ and $\langle 0,0,\frac 12,\frac 12,\frac12\rangle$ represented by the midpoints of $AB$ and $A'B$, respectively. These two points do not correspond to facet vectors and Lemma \ref{lem:free_space} implies that $yz$ is a free direction for $P$.

{\bf Case \ref{lem:pr-py-py-9}.10110}: $R\in [1,0,1,1,0]$. This case becomes identical to {\bf Case \ref{lem:pr-py-py-9}.00110} if we swap $X$ to $Y$ and $X'$ to $Y'$.

{\bf Case \ref{lem:pr-py-py-9}.01110}: $R\in [0,1,1,1,0]$. This case becomes identical to {\bf Case \ref{lem:pr-py-py-9}.00110} if we swap $X$ to $Z$ and $X'$ to $Z'$.

{\bf Case \ref{lem:pr-py-py-9}.00111}: $R\in [0,0,1,1,1]$. The midpoints of $RB$ and $AX'$ represent the same class $\langle 0,0,\frac 12,0,\frac 12\rangle\in \Z^5_{1/2}$. Using Lemma \ref{lem:translated_cell} for the dual cell $\mathcal D(x)$ and the face with dual cell $AXX'A'$ we get that $\mathcal D(x)$ contains a point $R'\in [0,0,1,1,0]$. From {\bf Case \ref{lem:pr-py-py-9}.00110} we conclude that $yz$ is a free direction for $P$.

{\bf Case \ref{lem:pr-py-py-9}.10111}: $R\in [1,0,1,1,1]$. The midpoints of $RB$ and $AY'$ represent the same class $\langle \frac12,0,\frac 12,0,\frac 12\rangle\in \Z^5_{1/2}$. Using Lemma \ref{lem:translated_cell} for the dual cell $\mathcal D(x)$ and the face with dual cell $AYY'A'$ we get that $\mathcal D(x)$ contains a point $R'\in [1,0,1,1,0]$. From {\bf Case \ref{lem:pr-py-py-9}.10110} we conclude that $P$ has a free direction.

{\bf Case \ref{lem:pr-py-py-9}.01111}: $R\in [0,1,1,1,1]$. The midpoints of $RB$ and $AZ'$ represent the same class $\langle 0,\frac12,\frac 12,0,\frac 12\rangle\in \Z^5_{1/2}$. Using Lemma \ref{lem:translated_cell} for the dual cell $\mathcal D(x)$ and the face with dual cell $AZZ'A'$ we get that $\mathcal D(x)$ contains a point $R'\in [0,1,1,1,0]$. From {\bf Case \ref{lem:pr-py-py-9}.01110}  we conclude that $P$ has a free direction

{\bf Case \ref{lem:pr-py-py-9}.1111*}: $R\in [1,1,1,1,0]$ or $R\in [1,1,1,1,1]$. This case is similar to {\bf Case \ref{lem:pr-pr-py-10}.1111*} of Lemma \ref{lem:pr-pr-py-10} as $\mathcal D(yz)=CXYZX'Y'Z'$ in both cases.
\end{proof}

\begin{lemma}\label{lem:pr-py-py}
If a triangular face $xyz$ of $P$ with prismatic dual cell has exactly two edges $xz$ and $yz$ with dual cells equivalent to pyramids over this triangular prism, then $P$ has a free direction.
\end{lemma}
\begin{proof}
Suppose $P$ does not have a free direction, then $\mathcal D(x)=BAXYZA'X'Y'Z'$ according to Lemma \ref{lem:pr-py-py-9}. The proof generally repeats the proof of Lemma \ref{lem:pr-pr-py} with minor changes. It might be useful to use Figure \ref{pict:pr-pr-py} to track dual cells of faces and edges we use in the arguments.

Parallelogram $XYY'X'$ is the dual cell of a face of the tiling $\mathcal T_P$. Let $G$ be the 3-dimensional face of $\mathcal T_P$ such that $\mathcal D(G)=XYY'X'$. In particular, triangle $F=xyz$ is a face of $G$. Let $H_{xy}$ be the face of $G$ adjacent to $F$ by $xy$. The dual cell of $H_{xy}$ contains the points $X$, $Y$, $Y'$ and $X'$ and is contained in $\mathcal D(xy)=AXYZA'X'Y'Z'$, hence $\mathcal D(H_{xy})=AXYA'X'Y'$ as this is the only 3-cell within $\mathcal D(xy)$ that contains $XYY'X'$ other than $\mathcal D(F)$. Since $\mathcal D(H_{xy})$ is equivalent to a triangular prism, then $H_{xy}$ is a triangle due to Lemma \ref{lem:triangle} unless $P$ has a free direction.

Let $H_{xy}=xyt$ for some point $t$. We look at the dual cell of the edge $xt$. This dual cell contains $AXYA'X'Y'$, the dual cell of $xyt$, and is contained in $BAXYZA'X'Y'Z'$, the dual cell of $x$ obtained in Lemma \ref{lem:pr-py-py-9}. The dual cell $\mathcal D(xt)$ does not contain $Z$ (or $Z'$) as in that case the copy of $P$ centered at $Z$ (or $Z'$) would contain all three vertices of $xyt$ but not the triangle $xyt$ itself. Hence, $\mathcal D(xt)$ contains $B$ and $\mathcal D(xt)=BAXYA'X'Y'$.

The dual cells $\mathcal D(xz)=BXYZX'Y'Z'$ and $D(xt)=BAXYA'X'Y'$ intersect by the dual 3-cell $BXYX'Y'$ and therefore the edges $xz$ and $xt$ belong to a two-dimensional face $H_{xz}$ of $G$ with the dual cell $\mathcal D(H_{xz})=BXYB'X'Y'$. 

We can use Lemma \ref{lem:pr-py-py-9} to find the dual cell $\mathcal D(y)$ as similarly to $x$, $y$ is a vertex of $xyz$ incident to edges having two non-equivalent dual cells. Thus, $\mathcal D(y)=CAXYZA'X'Y'Z'$. Similar arguments that we presented for finding $\mathcal D(xt)$ show that $\mathcal D(yt)=CAXYA'X'Y'$.

\begin{center}
\begin{figure}[!ht]
\includegraphics[width=0.8\textwidth]{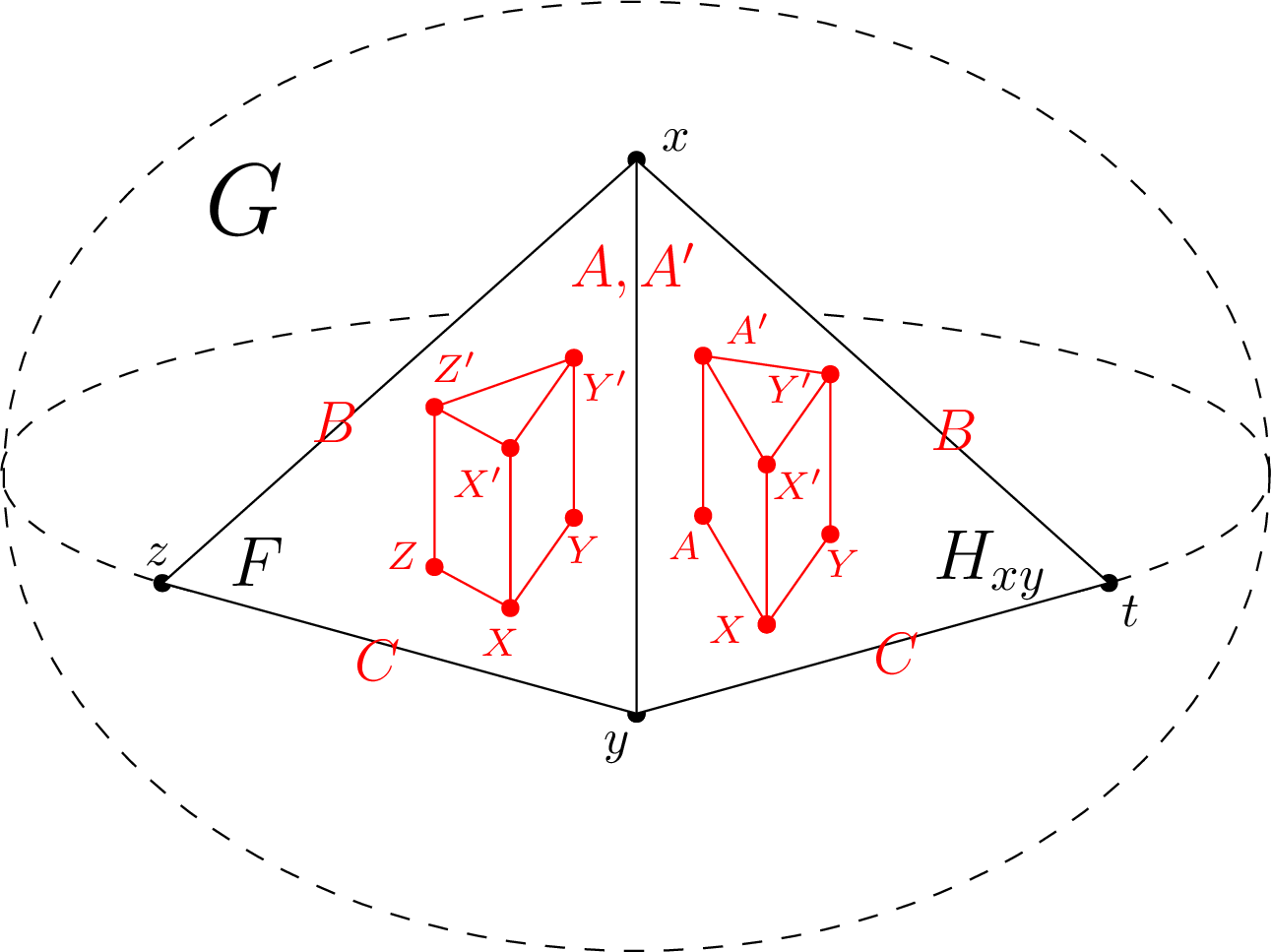}
\caption{An illustration for the proof of Lemma \ref{lem:pr-py-py}. The face $G$ with $\mathcal D(G)=XYY'X$ and its triangular faces $xyz$ and $xyt$ with prismatic dual cells. We put dual cells of two-dimensional faces inside corresponding triangles and show only additional points corresponding to edges. Dual cells are shown in red.}
\label{pict:pr-py-py}
\end{figure}
\end{center}

Similarly to the case of edges $xz$ and $xt$, the dual cells $\mathcal D(yz)=CXYZX'Y'Z'$ and $\mathcal D(yt)=CAXYA'X'Y'$ intersect by the dual 3-cell $CXYX'Y'$ and therefore $yz$ and $yt$ belong to a two-dimensional face $H_{yz}$ of $G$ with $\mathcal D(H_{yz})=CXYX'Y'$. Two 2-dimensional faces $H_{yz}$ and $H_{xz}$ have two vertices $z$ and $t$ in common, hence $zt$ is an edge of both and $H_{xz}=xzt$ and $H_{yz}=yzt$. This means that the face $G$ is the tetrahedron $xyzt$ as we identified four triangular faces $F=xyz$, $H_{xy}=xyt$, $H_{xz}=xzt$, and $H_{yz}=yzt$ of $G$.
{\sloppy

}

However the dual cell $\mathcal D(G)$ is $XYY'X'$, so $G$ is a contact face and must be centrally symmetric. Hence $G$ cannot be a tetrahedron.
\end{proof}

\section{Pyramid-Pyramid-Pyramid case}\label{sec:py-py-py}

In this section we assume that $P$ does not have a dual 3-cell equivalent to a cube. Also, if $F$ is a two-dimensional face of $P$ with the dual cell equivalent to a triangular prism, then $F$ is a triangle and dual cells of all edges of $F$ are pyramids over this prism. This is the last case with $pr(F)=0$ for all faces with prismatic dual 3-cells as in all other cases we establish that $P$ has a free direction and hence satisfies the Voronoi conjecture. 

For such $P$ we show that $P$ admits a canonical scaling or $P$ has a free direction. In both cases $P$ satisfies the Voronoi conjecture. The main tool we use to establish canonical scaling for $P$ is the gain function, see Definition \ref{def:gain}.  In this particular case we extend the notion of gain function for two facet vectors within a non-triangular dual 2-cell.

\begin{definition}\label{def:gain_new}
Let $KL$ and $LM$ be two facet vectors of a dual cell $KLMN$ equivalent to a parallelogram. If $O$ is point such that $OKLMN$ is a dual 3-cell, then we define
$$\gamma(KL,LM):=\gamma(KL,OL)\cdot \gamma(OL,LM).$$
This definition can also be extended to a sequence of facet vectors with each pair of consequent vectors within one dual 2-cell.

We note that this definition gives a way (maybe ambiguous) to define the gain function for each pair of appropriate facet vectors as every dual 2-face equivalent to a parallelogram belongs to a pyramidal dual 3-face as all parallelograms in a prismatic dual cell (dual cell of a triangle $xyz$) belong to pyramid subcells that are faces of pyramids over triangular prism (dual cells of the edges $xy$, $xz$, and $yz$).

We also note that this definition may give multiple values for the gain function $\gamma(KL,LM)$ if $KLMN$ is a subcell for two or more pyramidal dual 3-cells. We say that the parallelogram dual cell $KLMN$ is {\it coherent} dual cell if $\gamma(KL,LM)$ does not depend on the choice of $O$ for the cell $OKLMN$ equivalent to a pyramid over parallelogram.
\end{definition}

\begin{lemma}\label{lem:all_coherent}
If all dual 2-cells of $\mathcal T_P$ equivalent to parallelograms are coherent, then the Voronoi conjecture is true for $P$.
\end{lemma}
\begin{proof}
First we claim that the value of gain function $\gamma$ is 1 on every cycle of facet vectors of $\mathcal T_P$. It is enough to show that for cycles within single dual 3-cell of $\mathcal T_P$. For dual 3-cells equivalent to a tetrahedron or an octahedron, we refer to \cite[Lem 3.6]{GGM} which established that the gain function on a sequence of facet vectors around any vertex of a dual 3-cell is 1 provided all dual 2-subcells incident to this vertex are triangles. The claim for tetrahedral and octahedral dual 3-cells then follows as every relevant cycle of facet vectors within the dual 3-cell can be decomposed in (a product of) cycles around vertices and cycles defined by triangular subscells, and for each such cycle the gain function is~1 by the argument above or by Definition \ref{def:gain}.

For a dual 3-cell equivalent to a pyramid, the cycle of facet vectors around its apex has gain function 1 due to \cite[Lem. 3.6]{GGM}. For cycles around base vertices of the pyramid, the gain function is~1 due to Definition \ref{def:gain_new}. Similarly to the previous case, these five cycles around the apex and the base vertices together with the four cycles of facet vectors defined by triangular subcells of the pyramid can be used to decompose any cycle of facet vectors within the pyramid and give the claim of the lemma. 

The last case is a prismatic dual cell as $\mathcal T_P$ does not have cubical dual 3-cells. We fix one prismatic dual cell $XYZX'Y'Z'$ and show that all cycles within this cell have gain function~1. All the cycles within this cell are generated by cycles around vertices of the prism and the cycles defined by two triangular subcells of the prism, so it is enough to show that $$\gamma(XX',XY,YZ,XX')=1.$$ Here we use that edges $XX'$, $YY'$, and $ZZ'$ of the prism represent equivalent facet vectors of $\mathcal T_P$ so the value above is the value of $\gamma$ on the cycle around vertex $Y$ of the prism.

If $P$ does not have a free direction, then $XYZX'Y'Z'$ belongs to a 4-cell $AXYZX'Y'Z'$ equivalent to a pyramid over triangular prism. We use pyramids and tetrahedra within this 4-cell to extract the value of $\gamma(XX',XY,YZ,XX')$. From the dual 3-cell $AXYY'X'$ we know that $\gamma(XX',XY)=\gamma(XX',XA,XY)$. From the tetrahedral dual 3-cell $AXYZ$ we know that $\gamma(XA,XY,YZ)=\gamma(XA,AZ,YZ)$. And from the pyramidal dual 3-cell $AYZZ'Y'$ we know that $\gamma(AZ,YZ,XX')=\gamma(AZ,XX')$ because $XX'$, $YY'$ and $ZZ'$ represent equivalent facet vectors. Combining these equalities together we get
\begin{multline*}
\gamma(XX',XY,YZ,XX')=\gamma(XX',XA,XY,YZ,XX')=\\=\gamma(XX',XA,AZ,YZ,XX')=\gamma(XX',XA,AZ,XX').
\end{multline*}
The last quantity is 1 because $XX'-XA-AZ-XX'$ is a cycle within the pyramidal dual 3-cell $AXZZ'X'$

Once the gain function $\gamma$ has value 1 on every cycle of facet vectors of $\mathcal T_P$, then $\mathcal T_P$ admits a canonical scaling. In this case we fix a facet $F\in \mathcal T_P^4$ (the set of all facets of $\mathcal T_P$)  and set $s(F):=1$ where $s:\mathcal T_P^4\longrightarrow \R_+$ is the canonical scaling we construct. For a facet $G\in \mathcal T_P^4$ we choose any path $F=F_0,F_1,\ldots,F_m=G$ such that $F_i$ and $F_j$ share a face of codimension 2 and define 
$$s(G)=\gamma (F_0,F_1,\ldots,F_m).$$
It is easy to see that if all parallellograms are coherent then $s$ is indeed a canonical scaling for $\mathcal T_P$ and hence $P$ satisfies the Voronoi conjecture. 

We also refer to \cite{GGM} for more details on the connection between gain function and canonical scaling.
\end{proof}

It remains to prove that all dual 2-cells equivalent to parallelograms are coherent.

\begin{lemma}\label{lem:coherent}
If $G$ is a contact 3-dimensional face of $\mathcal T_P$, then $\mathcal D(G)$ is coherent or $P$ has a free direction.
\end{lemma}
\begin{proof}
We consider only the case when $P$ does not have a free direction.

Suppose that $G_1$ and $G_2$ are two two-dimensional faces of $\mathcal T_P$ incident to $G$ such that $\mathcal D(G_1)$ and $\mathcal D(G_2)$ are equivalent to pyramids over parallelograms. We need to show that $G_1$ and $G_2$ give rise to the same value of gain function between two facet vectors of $\mathcal D(G)$ using Definition \ref{def:gain_new}.

We note that no two two-dimensional faces of $G$ adjacent by an edge can both have dual cells equivalent to triangular prisms because the common edge of these two faces has dual cell equivalent to a pyramid over triangular prism (this is the case in this Section). However a pyramid over triangular prism has only one face equivalent to triangular prism.

$G$ is a 3-dimensional polytope and $G_1$ and $G_2$ are two-dimensional faces of $G$. We connect a vertex of $G_1$ with a vertex of $G_2$ by a path of edges of $G$. For every edge of this path, at least one of two incident two-dimensional faces of $G$ has dual cell equivalent to a pyramid, so it is enough to show that if $G_1$ and $G_2$ share a vertex, then they give rise to the same value of gain function between two facet vectors of $\mathcal D(G)$.

If $G_1$ and $G_2$ share an edge $e$, then $e$ does not belong to a two-dimensional face of $\mathcal T_P$ with dual 3-cell equivalent to a triangular prism. Indeed, in that case the dual cell $\mathcal D(e)$ would be a pyramid over triangular prism, but a pyramid over triangular prism does not have two pyramidal faces with a common base, and such two faces must be dual cells of $G_1$ and $G_2$. Thus, the two-dimensional faces that contain $e$ have tetrahedral, octahedral, or pyramidal dual 3-cells. Then $e$ is a locally ``Ordine'' edge meaning that the dual cell $\mathcal D(e)$ does not contain cubical or prismatic dual 3-cells as subcells. In that case the parallelogram dual cell of $G$ has the same gain function within cells corresponding to faces incident to $e$, see \cite[Sec. 7]{Ord}, in particular for the faces $G_1$ and $G_2$.

If $G_1\cap G_2$ is a vertex of $G$, then there is a cycle of two-dimensional faces of $G$ around this vertex, so there are two non-intersecting paths of two-dimensional faces of $G$ from $G_1$ to $G_2$, both with a common vertex $G_1\cap G_2$. If one of these paths does not contain two-dimensional faces with prismatic dual cells, then the faces $G_1$ and $G_2$ give rise to the same value of gain function between two facet vectors of $\mathcal D(G)$ as this value does not change if we travel along the path around $G_1\cap G_2$ using only two-dimensional faces with pyramidal dual 3-cells. Suppose that there is a face with prismatic dual cell on each path. It means that there are two triangular faces $H_1$ and $H_2$ of $G$ that share the vertex $V=G_1\cap G_2$ such that both $\mathcal D(H_1)$ and $\mathcal D(H_2)$ are equivalent to triangular prisms, see Figure \ref{pict:py-py-py}. 

\begin{center}
\begin{figure}[!ht]
\includegraphics[width=0.8\textwidth]{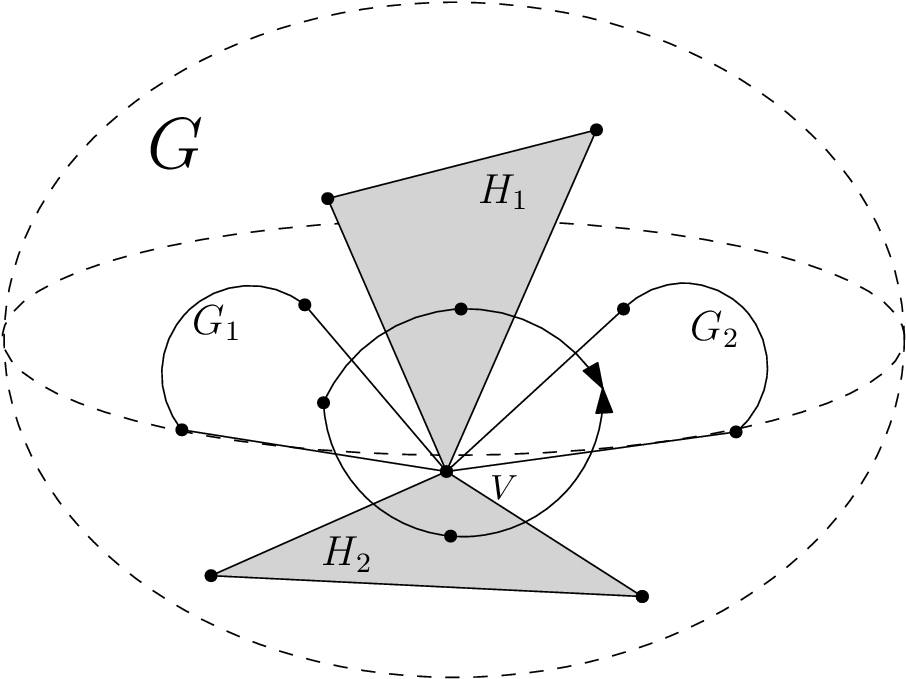}
\caption{An illustration for the proof of Lemma \ref{lem:coherent}. The face $G$ with $\mathcal D(G)=XYY'X$ and its faces $G_1$ and $G_2$ with a common vertex $V$. Two paths from $G_1$ to $G_2$ around $V$ pass through triangular faces $H_1$ and $H_2$ with prismatic dual cells.}
\label{pict:py-py-py}
\end{figure}
\end{center}

Let $XYY'X'$ be the dual cell of $G$ and let $XYZX'Y'Z'$ and $XYTX'Y'U$ be dual cells of $H_1$ and $H_2$. We can assume that $\overrightarrow{XX'}=\overrightarrow{YY'}=\overrightarrow{ZZ'}$ but $\overrightarrow{TU}$ may be equal to either $\overrightarrow{XX'}$ or $\overrightarrow{XY}$. 

We look at the parity class of the point $T$. The three-dimensional affine subspace of $\Z^5_p$ spanned by $XYZX'Y'Z'$ contains parity classes of $X$, $Y$, $Z$, $X+Y+Z$, $X'$, $Y'$, $Z'$, and $X'+Y'+Z'$. None of these classes can be the class of $T$ due to Lemmas \ref{lem:parity} and \ref{lem:3+1} as vertices of $XYZX'Y'Z'$ and $T$ are in the dual cell of $V$. Thus, there is an edge $e_A$ of $H_1$ with dual cell $AXYZX'Y'Z'$ such that $A$ and $T$ are in one affine subspace of $\Z^5_p$ parallel to the plane spanned by $XYZX'Y'Z'$. 

Let $\pi_A$ be the 4-dimensional linear subspace of $\Z^5_{1/2}$ spanned by the set of midpoints of $\mathcal D(e_A)=AXYZX'Y'Z'$. The set of midpoints of the pyramid $AXYZX'Y'Z'$ contains 14 classes in $\pi_A$. 8 of these classes are in the set of midpoints of $XYZX'Y'Z'$ and the remaining 6 classes correspond to the midpoints of facet vectors $AX$, $AY$, $AZ$, $AX'$, $AY'$, and $AZ'$. Among midpoints of the dual cell $XYTX'Y'U$, the midpoints of $TX$, $TY$, $TX'$, and $TY'$ contain two classes that correspond to contact faces of codimension $2$ (parallellogram subcells of $XYTX'Y'U$ that contain $T$). 4 midpoints of $TX$, $TY$, $TX'$, and $TY'$ together with 6 midpoints of $AX$, $AY$, $AZ$, $AX'$, $AY'$, and $AZ'$ give 8 classes (two classes are repeated twice) that form an affine subspace of $\pi_A$ parallel to the linear subspace spanned by 8 midpoints of $XYZX'Y'Z'$.

Now we can use Lemma \ref{lem:free_space} for the edge $e_A$ and the four-dimensional linear subspace $\pi_A$. Among 16 classes in $\pi_A$, 14 are the classes of midpoints of $\mathcal D(e_A)$ and two correspond to contact faces of codimension 2 of $XYTX'Y'U$. Thus, $e_A$ is a free direction for $P$ in that case which gives a contradiction.
\end{proof}

\begin{corollary}\label{cor:py-py-py}
If for every two-dimensional face $F$ of $P$ with prismatic dual 3-cell, dual cells of all edges of $F$ are equivalent to pyramids over triangular prisms, then $P$ satisfies the Voronoi conjecture.
\end{corollary}
\begin{proof}
If all dual 2-cells of $\mathcal T_P$ equivalent to parallelograms are coherent, then $P$ satisfies the Voronoi conjecture due to Lemma \ref{lem:all_coherent}. Otherwise, there is a contact 3-dimensional face of $P$ with incoherent dual 2-cell and $P$ has a free direction due to Lemma \ref{lem:coherent}. In this case $P$ satisfies the Voronoi conjecture as well.
\end{proof}

The proof of Lemma \ref{lem:coherent} above and the general approach for parallelohedra without dual 3-cells equivalent to prisms or cubes in Theorem \ref{thm:main} rely on the proof of Ordine \cite[Sec. 7]{Ord}. The most complicated part of the proof of Ordine and the only part that involves computer computations using {\tt PORTA} software is Case 4 in \cite[Section 7.6]{Ord}. In this particular case Ordine shows that there is no dual 4-cell (with all dual 3-cells equivalent to tetrahedra, octahedra, or pyramids) with incoherent parallelograms forming a family $\mathcal R$ such that
\begin{itemize}
\item each two parallelograms in $\mathcal R$ intersect over a vertex;
\item each vertex of a parallelogram in $\mathcal R$ belongs to at least one other parallelogram in $\mathcal R$.
\end{itemize}
In the five-dimensional case these computations can be avoided.

Particularly, if $e$ is an edge of a five-dimensional parallelohedron $P$ with dual 4-cell that contains a family of incoherent parallelograms satisfying the conditions above, then the first condition implies that $\mathcal D(e)$ contains two parallelograms $ABCD$ and $AXYZ$. For a certain choice of coordinate system in $\mathcal \Z^5_p$, the points $A, B, C, D, X,Y$, and $Z$ belong to the following parity classes
\begin{center}
\begin{tabular}{ll}
$A\in[0,0,0,0,0]$, & \\
$B\in[1,0,0,0,0]$,\qquad \qquad & $X\in[0,0,1,0,0]$,\\
$C\in[0,1,0,0,0]$,\qquad \qquad & $Y\in[0,0,0,1,0]$,\\
$D\in[1,1,0,0,0]$,\qquad \qquad & $Z\in[0,0,1,1,0]$.\\
\end{tabular}
\end{center}
In that case, the set of midpoints $M_{\mathcal D(e)}$ contains all points from the 4-dimensional space $x_5=0$ of $\Z^5_{1/2}$ and $e$ is a free direction of $P$ according to Lemma \ref{lem:free_space}

\section{Concluding remarks}\label{sec:final}

In this section we explain why our approach cannot be carried out in higher dimensions without significant improvement. Our approach relies on two results that seem to require additional elaboration in order to be used in dimensions 6 and beyond.

The first result is the classification of five-dimensional Dirichlet-Voronoi parallelohedra from \cite{5dim} and verification of the combinatorial condition from \cite{GGM} done in \cite{DGM-5dim} for every five-dimensional Dirichlet-Voronoi parallelohedra; we use that verification in the proof of Lemma \ref{lem:5-free} when referring to \cite[Thm. 1.3]{DGM-5dim}. More precisely, paper \cite{5dim} gives a list of 110~244 Dirichlet-Voronoi parallelohedra in $\R^5$ and verifies that this list results in a complete list of equivalence classes of five-dimensional Dirichlet-Voronoi parallelohedra (the classes are called {\it L-types} in \cite{5dim} and correspond to affinely different Delone subdivisions for lattices). Paper \cite{DGM-5dim} verifies that every five-dimensional Dirichlet-Voronoi parallelohedron from the list of 110~244 satisfies a combinatorial condition from \cite{GGM}; this condition is sufficient to claim the Voronoi conjecture for every parallelohedron equivalent to one from the list without assuming any geometric properties in addition to equivalence in the sense of Definition \ref{def:equiv}. While the verification of the condition from \cite{GGM} is computationally simple for a given parallelohedron, the full classification in $\R^6$ and beyond looks unreachable at this moment. Particularly, the paper \cite{SV} reports about more than 250~000 types of Delone triangulations (and consequently, primitive parallelohedra) in $\R^6$; a more recent paper \cite{BE} reports about more than 500~000~000 types of primitive parallelohedra in $\R^6$. Both computations were terminated before finding all triangulations/parallelohedra  and both suggest that the total number of parallelohedra in $\R^6$, both primitive and not, is too large for computational study without additional insight.

Moreover, even if such classification of six-dimensional parallelohedra will be obtained, the approach we used requires verification that every polytope in the anticipated list possesses combinatorics that enforces Voronoi conjecture. For the five-dimensional case this was done in \cite{DGM-5dim}. While for every separate parallelohedron this verification is simple, the enormous number of polytopes may make it computationally infeasible and require a more general combinatorial approach similar to the one developed in \cite{G-root}. 

We also mention a recent announcement of a complete classification of {\it $C$-types} of six-dimensional lattices by Dutour Sikiri\'c, Magazinov, and van Woerden in \cite{c-types}. The $C$-types correspond to affinely different one-dimensional skeletons of Delone tilings for lattices. This is a coarser classification compared to the classification of primitive parallelohedra but can be considered as an important first step towards it.

The second result is the classification of dual 3-cells by Delone \cite{Del}. In the five-dimensional case, dual 3-cells originate from two-dimensional faces that have a fairly simple structure that allowed us to prove many properties in Sections \ref{sec:prism} through \ref{sec:py-py-py}. In higher dimension, we would need to deal either with three-dimensional faces of parallelohedra with additional co-dimension in the spaces $\Lambda_p$ and $\Lambda_{1/2}$, or with dual 4-cells. However at this point there is no complete classification of dual 4-cells and, in particular, the question on dimension of affine space spanned by vertices of a dual 4-cell is still open.

For the conclusion we would like to formulate a few open questions and conjectures. The first one stems from Definition \ref{def:equiv}.

\begin{question}
Is it true that every two combinatorially equivalent parallelohedra are also equivalent in the sense of Definition \ref{def:equiv}?
\end{question}

Neither positive nor negative answer to this question affects the status of the Voronoi conjecture as there, could be two combinatorially equivalent Dirichlet-Voronoi polytopes with combinatorially distinct corresponding Delone tilings or there could be a combinatorially ``unique'' parallelohedron that does not satisfy the Voronoi conjecture. Nevertheless, we think that answering that question may give an additional insight on combinatorial restrictions imposed on parallelohedra.

As a final remark, we would like to formulate two well-known conjectures on dual cells. These conjectures are still open and having a counterexample for each of them will immediately give a counterexample to the Voronoi conjecture.

\begin{conjecture}[Dimension conjecture]
For every dual $k$-cell, the dimension of its affine span is equal to $k$.
\end{conjecture}

This conjecture is proved for $k\leq 3$ as all dual 3-cells are known. However, it is still open for $k\geq 4$. A stronger version of this conjecture imposes additional structure coming from Delone tilings.

\begin{conjecture}
For every dual $k$-cell $D$ there is a $k$-dimensional lattice $\Lambda$ such that there is a $k$-dimensional cell in the Delone tessellation of $\Lambda$ equivalent to $D$.
\end{conjecture}

\section*{Acknowledgments}

Initial versions of this work were prepared in collaboration with Alexander Magazinov. Unfortunately, he decided to step down as an author of the paper due to personal circumstances. The author is extremely grateful to Alexander for many fruitful discussions and shared ideas that resulted in this work.

Author's research is partially supported by the Alexander von Humboldt Foundation and by the Simons Foundation.

\appendix

\section{Engel's approach to the Voronoi conjecture}\label{sec:appendix}

In this section we briefly describe the approach of Peter Engel to the Voronoi conjecture in dimension 5 and explain why the results from \cite{Eng} and \cite{Eng00} do not give a complete proof of the five-dimensional case of the Voronoi conjecture. In some cases, the papers of Engel do not formulate precise definitions and do not provide rigorous proof for claimed properties. We tried our best to give these results interpretations or justifications in the language of this paper or other relevant literature.

We contacted Dr. Engel in May-June 2019, before the first version of this manuscript was posted, to clarify the issues that we identified, but unfortunately we have not received a reply that was able to resolve them. We value Dr. Engel's expertise in chemistry and crystallography, but we believe that mathematical contributions in \cite{Eng} and \cite{Eng00} do not result in a rigorous proof of the Voronoi conjecture in $\R^5$ or in sketch of such a proof.

The final judgement on the merits of the results of \cite{Eng} and \cite{Eng00} not related to the Voronoi conjecture is beyond the scope of this appendix.

\addtocontents{toc}{\SkipTocEntry}
\subsection*{Additional definitions}

We start by formulating a few more definitions that are crucial for understanding Engel's results. While some of these definitions can be re-worded using the language we use in the paper or are frequently used in the literature, we prefer to use the terms from \cite{Eng,Eng00} in this appendix.

\begin{definition}
Given a parallelohedron $P$, we define a {\it zone} $Z$ of $P$ is a class of parallel edges of $P$. The zone $Z$ is called {\it closed} if every two-dimensional face of $P$ has either $0$ or $2$ edges from $Z$. Otherwise, $Z$ is called an {\it open} zone. 
\end{definition}

Closed zone of parallelohedra can be used to identify some of free directions of $P$. If $e$ is a shortest edge of a closed zone $P$, then $e$ gives a free direction of $P$ and there exists a parallelohedron $Q$ such that $P=Q+I$ where $I$ is a segment of direction $e$. In other words, directions of closed zone can be separated as Minkowski summands of $P$. The condition of being part of closed zone is not necessary for that.

\begin{definition}
In the notations above, if $I$ is taken to be of length of the shortest edge of $Z$, then $Q$ is called a {\it contraction} of $P$ obtained by {\it zone contraction} process. 
\end{definition} 

In this case the zone defined by the same direction may become open or may completely disappear if $Q$ has no edges of that direction. In both cases $P$ and $Q$ have different combinatorial types and even if $P$ is a Voronoi parallelohedron, then $Q$ does not have to be within a combinatorial or affine type of Voronoi parallelohedron a priori. 

The second series of definitions concerns the structure of the decomposition of the cone of positive definite quadratic forms into the domains of existence. While the latter are never explicitly defined in \cite{Eng}, we think they are secondary cones defined in \cite{5dim}.

\begin{definition}
Let $Q$ be a positive definite quadratic form. The {\it domain of existence} of $Q$ consists of all quadratic forms $Q'$ that satisfy the following condition. For every positive definite quadratic form for we construct a basis with that Gram matrix, and then the lattice with basis. Then $Q'$ belong to the domain of existence of $Q$ if and only if, the Voronoi tilings of the lattices constructed for $Q'$ and $Q$ are equivalent in the sense of Definition \ref{def:equiv} with a stronger condition that we match vectors of the corresponding bases preserving the order.

Another way to think of that notion is that Delone tilings of lattices with Gram matrices $Q$ and $Q'$ are affinely equivalent under the linear map defined by the corresponding bases. The reduction theory of lattices according to their Delone tilings goes back to Voronoi \cite{Vor}.
\end{definition}

Each full-dimensional domain of existence or, equivalently, secondary cone corresponds to Delone triangulations and result from primitive parallelohedra because a small perturbation of the corresponding Gram matrix does not change the combinatorial type of Delone triangulation. On the other hand, for non-primitive Voronoi parallelohedra the secondary cones are not full-dimensional.

Each secondary cone is a convex pointed cone within the cone of positive definite quadratic forms, see \cite{5dim} for more details. The facets of full-dimensional cones are given by linear equations on the coefficients of the corresponding quadratic forms. The equations are called {\it Voronoi regulators} in \cite{5dim}, and the facets are called {\it walls} in \cite{Eng}, but we note that there is no rigorous definition for that notion in \cite{Eng}. However, the walls are characterized by linear conditions on coefficients of quadratic forms coming from more than $d$ facets of a Voronoi parallelohedron incident to one vertex, and these linear conditions coincide with the conditions given by Voronoi regulators.

\addtocontents{toc}{\SkipTocEntry}
\subsection*{Engel's approach}

The main idea of the approach of Engel is in checking that in a given dimension, for a given parallelohedron $P$, every contraction $P$ is a (combinatorially) Voronoi parallelohedron by finding an appropriate Gram matrix for the associated lattice. The result that every contraction of a Voronoi parallelohedron satisfies the Voronoi conjecture is formulated in Lemma \ref{lem:p+i} and was proved by Grishukhin \cite{Gri2} and V\'egh \cite{Vegh-contraction}. Note that this is not enough to claim the Voronoi conjecture in any given dimension unless there is a proof that every parallelohedron can be obtained from a Voronoi parallelohedron by repeated zone contraction process.

While the detailed algorithm in $\R^5$ is not described in \cite{Eng}, our interpretation is the following. The algorithm starts from 222 primitive parallelohedra; we note that the 222nd primitive parallelohedron was first identified in \cite{Eng} while Ryshkov and Baranovski \cite{BR} found only 221 primitive parallelohedra in $\R^5$. For every primitive parallelohedron $P$, the algorithm searches for a closed zone and contracts it. The process of zone contraction repeats till the new parallelohedron has no closed zones at all, such parallelohedra are called {\it totally zone contracted}, or till any further contraction will reduce the dimension of the polytope, such parallelohedra are called {\it relatively zone contracted}. 

If we start from a Voronoi parallelohedron $P=Q+I$, then Engel claims that for contracted parallelohedron $Q$, the quadratic form for the associated lattice can be obtained from the one of $P$ by subtracting the rank one form $v\cdot v^t$ where $v$ is the vector representing segment~$I$, cf. \cite[Sect. 5]{5dim}. This is true if $Q$ is known to be a Voronoi parallelohedron, but Engel was able to verify it independently for every maximal sequence of contractions resulting in totally or relatively zone contracted parallelohedron.

After that Engel claims that since the list of all obtained totally contracted parallelohedra is complete, and each of them is combinatorially equivalent to a Voronoi parallelohedron, then so are their extensions obtained by the operation opposite to contraction. This should imply that ``Every parallelohedron in $\R^5$ is combinatorially equivalent to a Voronoi parallelohedron''. This claimed as the main result of \cite{Eng}.

\addtocontents{toc}{\SkipTocEntry}
\subsection*{Logical gaps in Engel's approach}

Unfortunately, we think that there at least two logical gaps in Engel's approach.

The first one is minor as Engel's proof that ``intermediate'' contractions of primitive parallelohedra are Voronoi parallelohedra relies on the property that (repeated) extensions of zone contracted Voronoi parallelohedra satisfy the Voronoi conjecture. Engel does not provide any evidence of computations related to that property. However, the property holds for any dimension and can be established using properties of free directions that we described in Section \ref{sec:freedir}.

The second gap is much more important and in our opinion makes the claimed main result of \cite{Eng} unproven without possibility to fill the gap. Engel's claim that his approach obtains all totally contracted parallelohedra is not justified and there is no attempt in \cite{Eng00} to justify it. Particularly, it is unclear why every totally zone contracted parallelohedron in $\R^5$ can be obtained from one of 222 primitive parallelohedra by several zone contractions. Without proving that claim, it is possible that there exists a five-dimensional parallelohedron $P$ that cannot be obtained by zone contractions from primitive ones and does not satisfy the Voronoi conjecture.

Moreover, in six dimensions the Voronoi parallelohedron $P$ of the dual root lattice $\mathsf E_6^*$ has the following property: it is totally contracted and does not have a free direction; this example is mentioned in Engel's paper \cite{Eng}. This means that $P$ cannot be obtained by zone contraction from any other six-dimensional parallelohedron.

Existence of a five-dimensional parallelohedron with similar properties will invalidate Engel's arguments right away, and there is no attempt in \cite{Eng} to prove that there is no such a parallelohedron.

\addtocontents{toc}{\SkipTocEntry}
\subsection*{Further expansion of Engel's algorithm}

We also briefly describe results of Engel from \cite{Eng00}. Notably, the paper \cite{Eng00} claims that \cite{Eng} contains a proof of the Voronoi conjecture as stated in the present paper using affine equivalence rather than combinatorial equivalence explicitly claimed in \cite{Eng}.

The main results claimed in \cite{Eng00} concern classification of contraction and subordination types of five-dimensional parallelohedra. More specifically, Engel starts from the list of relatively and totally zone-contracted parallelohedra and uses the operations of zone extensions and zone contractions to obtain all possible parallelohedra that are subdivided into contraction classes and into classes according their subordination symbols. The latter classification is a coarser version of combinatorial equivalence for polytopes.

We want to emphasize that the paper \cite{Eng00} does not make an attempt to justify that the classification of totally and relatively zone-contracted parallelohedra from \cite{Eng} is complete assuming that for further computations.

Thus, the paper \cite{Eng00} does not fill the gap we identified above. We also refer to \cite{5dim} for a detailed analysis of computational results from \cite{Eng00} and for comparison of said results with the classification of five-dimensional Voronoi parallelohedra obtained \cite{5dim}.

\addtocontents{toc}{\SkipTocEntry}
\subsection*{Concluding remarks}

Unlikely to the approach of Engel from \cite{Eng,Eng00}, our proof in the present paper does not rely on totally contracted parallelohedra. Instead of that, for every five-dimensional parallelohedron $P$, we prove that $P$ has a free direction, and then we rely on complete classification of five-dimensional Voronoi parallelohedra from \cite{5dim} and their combinatorial properties, or $P$ has a canonical scaling, and satisfies the Voronoi conjecture as well.

Concluding the discussions of Engel's contributions, we would like to reiterate that in our opinion, his approach relies on a classification of totally zone contracted parallelohedra in $\R^5$ but there is no justification that such a classification is complete or even feasible without some additional assumptions. Without such a justification, we think that approach of \cite{Eng} and \cite{Eng00} contains an irreparable gap.

\end{document}